\documentclass[11pt]{amsart}

\usepackage[margin=3cm]{geometry}
\usepackage{amsfonts}
\usepackage{amsmath}
\usepackage{amsthm}
\usepackage{amssymb}
\usepackage{mathrsfs}
\usepackage{slashed}
\usepackage{lipsum}
\usepackage{newclude}
\usepackage{mathtools}
\usepackage{physics}
\usepackage{nomencl}
\makenomenclature
\usepackage{varioref}
\usepackage{hyperref}
\usepackage{cleveref}
\usepackage{color}
\usepackage{comment}
\usepackage{array}   
\newcolumntype{L}{>{$}l<{$}} 

\numberwithin{equation}{section}

\newtheorem{thm}{Theorem}[section]
\newtheorem{prop}[thm]{Proposition}
\newtheorem{cor}[thm]{Corollary}

\newtheorem{lem}[thm]{Lemma}

\newtheorem{defn}[thm]{Definition}

\newtheorem{rem}[thm]{Remark}
\newtheorem{ex}[thm]{Example}

\def\FT{{\mathscr F}}
\newcommand{\inlinemaketitle}{{\let\newpage\relax\maketitle}}
\newcommand{\be}{\begin{equation}}
\newcommand{\ee}{\end{equation}}
\def\eps{\varepsilon}
\def\HS{{\mathtt{HS}}}
\def\op{{\mathtt{op}}}
\def\TT{{\mathbb T}}

\def\NN{{\mathbb N}}
\def\dualSU2{\frac12\NN_0}
\def\RR{{\mathbb R}}
\def\ZZ{{\mathbb Z}}

\def\jp#1{{\left\langle{#1}\right\rangle}}
\def\C{{\mathbb C}}
\def\H{{\mathcal{H}}}
\def\Gh{{\widehat{\G}}}

\newcommand{\CM}{{\mathcal{M}}}
\newcommand{\CR}{{\mathcal{R}}}

\def\C{{\mathbb C}}

\newcommand{\CQ}{{\mathcal{Q}}}

\newcommand{\tens}{\otimes}
\renewcommand{\o}{{}_{(1)}}
\renewcommand{\t}{{}_{(2)}}
\renewcommand{\th}{{}_{(3)}}

\newcommand{\extd}{{\rm d}}
\newcommand{\del}{{\partial}}
\newcommand{\la}{{\triangleright}}
\newcommand{\ra}{{\triangleleft}}
\newcommand{\isom}{{\cong}}

\renewcommand{\>}{{\rangle}}

\def\dpi{{d_\pi}}

\def\npi{{n_\pi}}
\def\FT{{\mathscr F}}

\def\Ad{\mathrm{Ad}}

\def\A{\mathcal{A}}

\def\SUq2{{\rm SU^{\,q}_2}}
\def\SU2{{\rm SU_2}}
\def\SL2{{\rm SL_2}}
\def\su2{\mathfrak{{ su_2}}}
\def\sl2{\mathfrak{{ sl_2}}}
\def\G{{\mathbb{G}}}
\def\D{{\mathcal{D}}}

\DeclareMathOperator{\Dom}{Dom}

\DeclareMathOperator{\Id}{Id}
\DeclareMathOperator{\id}{id}
\DeclareMathOperator{\diag}{diag}

\DeclareMathOperator{\Span}{span}

\DeclareMathOperator{\sign}{sign}

\begin{document}

\title[Smooth dense subalgebras and Fourier multipliers]{Smooth dense subalgebras and Fourier multipliers on compact quantum groups}
\author{Rauan Akylzhanov, Shahn Majid and Michael Ruzhansky}
\thanks{The first author was partially supported by the Simons - Foundation grant 346300 and the Polish Government MNiSW 2015-2019 matching fund.
The third author was supported in parts by the EPSRC
 grant EP/K039407/1 and by the Leverhulme Grant RPG-2014-02. 
 No new data was created or generated during the course of this research.
}
\maketitle
\begin{abstract}

We define and study dense Frechet subalgebras of compact quantum groups consisting of elements rapidly decreasing with respect to an unbounded sequence of real numbers.
Further, this sequence can be viewed as the eigenvalues of a Dirac-like operator and we characterise the boundedness of its commutators in terms of the eigenvalues. Grotendieck's theory of topological tensor products immediately yields a Schwartz kernel theorem for linear operators on compact quantum groups and allows us to introduce a natural class of pseudo-differential operators on compact quantum groups. 
As a by-product, we develop elements of the distribution theory and corresponding Fourier analysis. We give applications of our construction to obtain sufficient conditions for $L^p-L^q$ boundedness of coinvariant linear operators.  
We provide necessary and sufficient conditions for algebraic differential calculi on Hopf subalgebras of compact quantum groups to extend to the proposed smooth structure.
We check explicitly that these conditions hold true on the quantum $\SUq2$ for both its 3-dimensional and 4-dimensional calculi.
\end{abstract}

\section{Introduction}

In \cite{HL1936} Hardy and Littlewood proved the following generalisation of the Plancherel's identity
on the circle $\TT$, namely
\begin{equation}
\label{H_L_inequality}
	\sum_{m \in \ZZ}{(1+|m|)}^{p-2}|\widehat{f}(m)|^{p} \leq C_p\|f\|^p_{L^p(\TT)},
	\quad
	1<p\leq 2.
\end{equation}
Hewitt and Ross \cite{HR1974} generalised this to the setting of compact abelian groups. Recently, the inequality has been extended \cite{ANR2015} to compact homogeneous manifolds. In particular, on a compact Lie group $G$ of topological dimension $n$, the result can be written as 
\begin{equation}
\label{H_L_inequality-alt-G2}
\sum_{\pi\in\widehat{G}} 
\dpi^{p(\frac2p-\frac12)}
 \jp{\pi}^{n(p-2)} 
\|\widehat{f}(\pi)\|_{\HS}^p \leq
   C_p\|f\|_{L^p(G)}^p,
\end{equation}
where $\jp\pi$ are obtained from eigenvalues of the Laplace operator $\Delta_{G}$ on  $G$ by
$$
\sqrt{I-\Delta_{G}}\pi_{ij}=\jp\pi \pi_{ij},\quad i,j=1,\ldots,\dpi.
$$

In \cite{Youn2017} the Hardy-Littlewood inequality has been extended to compact matrix quantum groups of Kac type. For this purpose, the author introduced a natural length function and extended the notion of 'rapid decay' to compact matrix quantum groups of Kac type. However, there are many non-Kac compact quantum group. For example,  $\SUq2$ is not of Kac type. On the other hand, the inequality \eqref{H_L_inequality-alt-G2} on a compact Lie group $G$ can be given a differential formulation as 
\begin{equation}
\label{H_L_inequality-alt-G}
\| \FT_{G}{(1-\Delta_{G})^{n(\frac12-\frac1p)} f}\|_{\ell^p(\widehat{G})} \leq
   C_p\|f\|_{L^p(G)},
\end{equation}
where $\Delta_{G}$ is the Laplacian on $G$ and $\FT_{G}$ is the group Fourier transform.
In the view of \cite{Lichnerowicz1990}, the operator $I-\Delta_{G}$ is the square of the spinor Dirac operator restricted to smooth functions on $G$. Thus, one can view the identity as associated to a `Dirac-like' operator understood broadly.

The Dirac operator was first introduced in 1928 by Paul Dirac to describe the evolution of fermions and bosons and plays an essential role in mathematical physics and representation theory.  The geometric Dirac operator $D$ can be constructed on arbitrary spin Riemannian manifold $(M,g)$ and Alain Connes showed \cite{Connes2013} that most of the geometry of $(M,g)$ can be reconstructed from the Dirac operator characterised abstractly as a `spectral triple' $(C^{\infty}(M),D,L^2_{spin}(M))$. The axioms of a spectral triple in Connes' sense come from KO-homology. However, it is known that the $q$-deformed quantum groups and homogeneous spaces do not fit into Connes axiomatic framework\cite{Connes1995} if one wants to have the correct classical limit and various authors have considered modification of the axioms. Another problem is that of `geometric realisation' where a given spectral triple operator (understood broadly) should ideally have an interpretation as built from a spin connection and Clifford action on a spinor bundle. A unified algebro-geometric approach to this has been proposed in \cite{BeggsMajid2015} starting with a differential algebra structure on a possibly noncommutative `coordinate algebra' and building up the geometry layer by layer so as to arrive at a noncommutative-geometrically realised $D$ as an endpoint. 

Compact quantum groups are quantisations of Poisson Lie groups and it is natural to expect that every compact quantum group should possess a spectral triple $(\mathcal{A},\mathcal{H},\mathcal{D})$ by a quantisation process of some sort. Following an approach suggested by \cite{ConnesLandi2001}, Chakraborty and Pal constructed \cite{Chakraborty2008} a spectral triple on the quantum $\SUq2$. 
A Dirac operator agreeing with a real structure on $\SUq2$  has been suggested in \cite{Dabrowski2005}, which required a slight modification on the spectral triple axioms. More recently, inspired in part by \cite{Fiore1998}, Nesheveyev and Tuet constructed \cite{Neshveyev2010} spectral triples on the $q$-deformation $G_q$ of a compact simply connected Lie group $G$.  To the best of our knowledge, it seems to be an open question whether there exists a spectral triple on an arbitrary compact quantum groups. And there remains the question of linking proposed spectral triples to the geometric picture.  At the root of this problem is how to marry the analytic considerations of compact quantum groups to the differential-algebraic notion of differential calculus in the more constructive approach. 

An outline of the results is as follows. After the set-up in Section~2 of quantum group Fourier transform, Paley-type inequalities in compact quantum groups $\G$ are developed in Section~3 following the classical case in \cite{ANR2015}. Section~4 studies left Fourier multipliers $A$ (i.e. translation coinvariant operators) with associated symbol $\sigma_A$. In Section~5 we introduce the notion of a summable operator $\D:L^2(\G)\to L^2(\G)$ defined by a sequence of eigenvalues according to the Peter-Weyl decomposition of $\G$ and subject to a summability condition. This allows us to define a smooth domain $C^{\infty}_{\D}=\bigcap\limits_{\alpha>0}\Dom(\left|\D\right|^{\alpha})$ as a ``smooth'' subspace of $\G$ since $\Dom(\left|\D\right|)\subset \C[\G]\subset \G$. We then show in Theorem \ref{THM:D-a-bounded} that $(C^{\infty}_{\D},L^2(\G),\D)$ obeys a minimal notion of `bare spectral triple' which is inspired by and expresses a part of Connes' axioms related to bounded commutators. 
As a consequence, we obtain (see Theorem \ref{THM:D-a-bounded})
$$
\| \FT_{\G}\left|\D\right|^{{\beta}(\frac12-\frac1p)}f\|_{\ell^p(\Gh)} \leq
   C_p\|f\|_{L^p(\G)},\quad 1<p\leq 2,
$$
where $\beta=\inf \{s>0\colon \Tr \left|\D\right|^{-s}<\infty\}$ is the spectral dimension \cite{Connes1995} of $C^{\infty}_{\D}$. 
 This includes the example of \cite{Chakraborty2010} on $\SUq2$ with eigenvalues $\pm(2l+1)$ in the spin $l$ part of the decomposition. 
 
Section 6 studies elements of distributions and rapid decay using $C^\infty_\D(\G)$ and introduces a notion of pseudo-differential operator in this context. 

Finally, Section~7 looks at how this notion of $C^\infty_\D$ relates to the algebraic notion of differential 1-forms in the algebraic side of noncommutative differential geometry. We show that both the standard 3D left covariant and 4D bi-covariant differential calculi on $\C[\SUq2]$ in \cite{Wor89} extend to $C^\infty_\D$ if we take $\D$ with eigenvalues $\pm[2l+1]_q$ where $[n]_q=(q^n-q^{-n})/(q-q^{-1})$ is a $q$-integer. Thus our approach to `smooth functions' is compatible with these $q$-differential calculi, marrying the analytic and algebraic approaches. This $q$-deformed choice of $\D$ no longer obeys the bounded commutators condition for a bare spectral triple but is a natural $q$-deformation of our previous choice. On the other hand it is more closely related to the natural $q$-geometrically-realised Dirac\cite{Majid2003} and square root of a Laplace \cite{Majid2015} operators on $\SUq2$ which similarly have eigenvalues modified via $q$-integers. 

The authors wish to thank Yulia Kuznetsova for her advice and comments.

\section{Preliminaries}

The notion of compact quantum groups has been introduced by Woronowicz in \cite{Woronowicz1987-b}. 	Here we adopt the defintion from \cite{Woronowicz1998}.
\begin{defn} A compact quantum group is a pair $(\G,\Delta)$ where $\G$ is a unital $C^*$-algebra, $\Delta\colon \G\to \G\otimes\G$ is a unital, $\star$-homomorphic map which is coassociative, i.e.
$$
(\Delta\otimes\Id_{\G})\circ\Delta=(\Id_{\G}\otimes\Delta)\circ\Delta
$$
and
$$
\overline{\Span\{(\Id_{\G}\otimes \G)\Delta(\G)\}}
=
\overline{\Span\{(\G\otimes\Id_{\G})\Delta(\G)\}}=\G\otimes\G,
$$
where $\G\otimes \G$ is a minimal $C^*$-tensor product.
\end{defn}
The map $\Delta$ is called the coproduct of $\G$ and it induces the convolution on the predual $L^1(\G)$
$$
\lambda\ast\mu:=(\lambda\otimes\mu)\circ\Delta,\quad \lambda,\mu\in L^1(\G).
$$
\begin{defn} Let $(\G,\Delta)$ be a compact quantum group. A finite-dimensional representation $\pi$ of $(\G,\Delta)$ is a matrix $[\pi_{ij}]$ in $M_{n}(\G)$ for some $n$ such that
\begin{equation}
\Delta\pi_{ij}=\sum\limits^{n}_{k=1}\pi_{ik}\otimes\pi_{kj}
\end{equation}
for all $i,j=1,\ldots,n$. We denote by $\widehat{\G}$ the set of all finite-dimensional irreducible unitary representations of $(\G,\Delta)$.
\end{defn}
Here we denote by $M_n(\G)$ the set of $n$-dimensional matrices with entries in $\G$.
Let $\C[\G]$ denote the Hopf subalgebra space of $\G$ spanned by the matrix elements $\pi_{ij}$ of finite-dimensional unitary representations $\pi$ of $(\G,\Delta)$. 
It can be shown \cite[Proposition 7.1, Theorem 7.6]{MD1998} that $\C[\G]$ is a Hopf $*$-algebra dense in $\G$. Every element $f\in \C[\G]$ can be expanded in a finite sum
$$
f
=
\sum\limits_{\pi\in I_f}\sum\limits^{n_{\pi}}_{i,j=1}c_{ij}\pi_{ij},\quad \text{ $I_f$ is a finite index set}.
$$
It is sufficient to define the Hopf $*$-algebra structure on $\C[\G]$ on generators $\pi_{ij}$ as follows
$$
\eps(\pi_{ij})=\delta_{ij},\quad S(\pi_{ij})=\pi^*_{ji} \quad \text{for}\quad  \pi\in\Gh,i,j=1,\ldots,n_{\pi},
$$
where $\eps\colon \C[\G]\to \C$ is the counit and $S\colon \C[\G]\to\C[\G]$ is the antipode.
These operations satisfy usual compatiblity conditions with coproduct $\Delta$ and product $m_{\G}$.

Every compact quantum group possesses \cite{Dijkhuizen1994} a functional $h$ on $\G$ called the Haar state such that 
$$
(h\otimes\Id_{\G})\circ\Delta(a)=h(a)1=(\Id_{\G}\otimes h)\circ\Delta(a).
$$
For every $\pi\in\Gh$ there exists a positive invertible matrix $Q_{\pi}\in \mathbb{C}^{n_{\pi}\times n_{\pi}}$ which is a unique intertwiner $Hom(\pi,\pi^*)$ such that 
\begin{equation}
\label{EQ:Tr-Q-pi}
\Tr Q_{\pi}= \Tr Q^{-1}_{\pi}>0.
\end{equation}
We can always diagonalize matrix $Q_{\pi}$ and therefore we shall write
$$
Q^{\pi}=\diag(q^{\pi}_1,\ldots, q^{\pi}_{n_{\pi}}).
$$
It follows from \eqref{EQ:Tr-Q-pi} that
$$
\sum\limits^{n_{\pi}}_{i=1}q^{\pi}_i
=
\sum\limits^{n_{\pi}}_{i=1}\frac1{q^{\pi}_i}
=:\dpi,
$$
which defines the quantum dimension $\dpi$ of $\pi$.
If $\G$ is a compact Kac group, then $\dpi=n_{\pi}$.
The Peter-Weyl orthogonality relations are as follows
\begin{eqnarray}
\begin{aligned}
\label{EQ:peter-weyl-orth}
h((\pi_{ij})^*\pi'_{kl})&=\delta_{\pi\pi'}\delta_{ik}\delta_{jl}\frac1{\dpi q^{\pi}_k},
\\
h(\pi_{kl}(\pi'_{ij})^*)&=\delta_{\pi\pi'}\delta_{ik}\delta_{jl}\frac{q^{\pi}_j}{\dpi}.
\end{aligned}
\end{eqnarray}
The quantum Fourier transform $\FT_{\G}\colon L^1(\G)\to L^{\infty}(\Gh)$ is given by
\begin{equation}
\widehat{f}(\pi)_{ij}=h(f \pi^*_{ji}),\quad i,j=1,\ldots, n_{\pi},
\end{equation}
where $L^1(\G)$ and $L^{\infty}(\Gh)$ are defined below.
The inverse Fourier transform $\FT^{-1}_{\G}$ is given by
$$
f=
\sum\limits_{\pi\in\widehat{\G}}
\dpi\Tr((Q^{\pi})^{-1}\pi  \widehat{f}(\pi))=\sum_\pi \sum^{\npi}_{i,j=1}{d_\pi\over q^\pi_i}\pi_{ij}\hat f(\pi)_{ji}.
$$

From this it follows that
$\{\sqrt{\dpi q^{\pi}_i} \pi_{ij}\colon \pi\in\Gh, \, 1\leq i,j\leq n_{\pi}\}$ is an orthonormal basis in $L^2(\G)$.
The Plancherel identity takes the form
\begin{equation}
\label{EQ:plancherel-0}
(f,g)_{L^2(\G)}
=
\sum\limits_{\pi\in\Gh}\dpi
\Tr 
\left(
(Q^{\pi})^{-1} \widehat{f}(\pi)\widehat{g}(\pi)^*
\right).
\end{equation}

We denote by $\mathcal{C}(\pi)$ the coefficients subcoalgebra
$$
\mathcal{C}(\pi)=\Span\{\pi_{ij}\}^{n_{\pi}}_{i,j=1}.
$$
The Peter-Weyl decomposition on the Hopf algebra $\C[\G]$ is of the form
\begin{equation}
\label{EQ:Peter-Weyl-decomposition}
\C[\G]
=
\bigoplus\limits_{\pi\in\Gh}\mathcal{C}(\pi).
\end{equation}
Let $L^2(\G)$ be the GNS-Hilbert space associated with the Haar weight $h$.
We denote by $L^{\infty}(\G)$ the universal von Neumann enveloping algebra of $\G$.
The co-product $\Delta$ and the Haar weight $h$ can be uniquely extended to $L^{\infty}(\G)$. 
In general, there are two approaches to locally compact quantum groups: $C^*$-algebraic and von Neumann-algebraic. 

Let $\psi$ be a normal semi-finite weight on the commutant $[L^{\infty}(\G)]^{!}$ of the von Neumann algebra $L^{\infty}(\G)$. Let $L^1(\G,\psi)$ be the set of all closed, densely defined operators $x$ with polar decomposition $x=u\left|x\right|$ such that there exists positive $\phi\in L^{\infty}(\G)_*$ and its spatial derivative $\frac{d\phi}{d\psi}=\left|x\right|$. Setting $\|x\|_{L^1(\G)}=\phi(1)=\|\phi\|_{L^{\infty}(\G)_*}$ yields an isometric isomorphism between $L^1(\G,\psi)$ and $L^{\infty}(\G)_*$. 
Analogously, we denote by $L^p(\G,\psi)$ the set of all closed, densely defined operators $x$ such that there exists $\phi\in L^1(\G,\psi)$ such that $\left|x\right|^p=\frac{d\phi}{d\psi}$ with the $L^p$-norm given by $\|x\|_{L^p(\G)}=\phi(1)^{\frac1p}$.
These spaces are isometrically isomorphic to the Haagerup $L^p$-spaces  \cite{Haagerup1977} and are thus independent of the choice of $\psi$.

One can introduce the Lebesgue space $\ell^p(\Gh)$ on the dual $\Gh$ as follows
\begin{defn} We shall denote by $\ell^p(\Gh)$ the space of sequences $\{\sigma(\pi)\}_{\pi\in\Gh}$ endowed with the norm
\begin{equation}
\|\sigma\|_{\ell^p(\Gh)}
=
\left(
\sum\limits_{\pi\in\Gh}
\dpi\npi
\left(
\frac{
\|\sigma(\pi)\|_{\HS}
}
{
\sqrt{\npi}
}
\right)^p
\right)^{\frac1p},\quad 1\leq p<\infty.
\end{equation}
Here by the Hilbert-Schmidt norm we mean
\begin{equation}
\label{EQ:HS-norm}
\|\sigma(\pi)\|_{\HS}
=
\Tr (Q^{\pi})^{-1}\sigma(\pi)\sigma(\pi)^*.
\end{equation}
For $p=\infty$, we write $L^{\infty}(\Gh)$ for the space of all $\sigma$ such that
\begin{equation}
\|\sigma\|_{L^{\infty}(\Gh)}
:=\sup_{\pi\in\Gh}\frac{\|\sigma(\pi)\|_{\HS}}{\sqrt{\npi}}<\infty.
\end{equation}
\end{defn}
It can be shown that these are interpolation spaces in analogy to similar family of spaces on compact topological groups introduced in \cite{RuzhanskyTurunen10}.
In this notation we can rewrite \eqref{EQ:plancherel-0} as follows
\begin{equation}
\label{EQ:plancherel-1}
\|f\|^2_{L^2(\G)}
=
\sum\limits_{\pi\in\Gh}\dpi
\|\widehat{f}\|^2_{\HS}.
\end{equation}

It can be shown that $\FT_{\G}\colon f\mapsto \widehat{f}=\{\widehat{f}(\pi)\}$ is a contraction, i.e.
\begin{equation}
\|\widehat{f}(\pi)\|_{\op}
\leq
\|f\|_{L^1(\G)},\quad \pi\in\Gh.
\end{equation}
Using the Hilbert-Schmidt  norm and unitarity $\pi\pi^*=\Id$ of $\pi$'s, one can show
\begin{equation}
\|\widehat{f}(\pi)\|_{\HS}
\leq
\sqrt{\npi}
\|f\|_{L^1(\G)},\quad \pi\in\Gh.
\end{equation}
Hence, by the interpolation theorem for $1<p\leq 2$ we obtain two versions Hausdorff-Young inequality
\begin{eqnarray}
\label{EQ:H-Y}
\left(
\sum\limits_{\pi\in\Gh}\dpi \|\widehat{f}(\pi)\|^{p'}_{
\ell^p(\C^{n_{\pi}\times n_{\pi}})
}
\right)^{\frac1{p'}}
=:
\|\widehat{f}\|_{\ell^{p'}_{sch}(\Gh)}
\leq
\|f\|_{L^p(\G)},
\\
\label{EQ:H-Y-2}
\left(
\sum\limits_{\pi\in\Gh}
\dpi\npi
\left(\frac{\|\widehat{f}(\pi)\|_{\HS}}{\sqrt{\npi}}\right)^{p'}
\right)^{\frac1{p'}}
=:
\|\widehat{f}\|_{\ell^{p'}(\Gh)}
\leq
\|f\|_{L^p(\G)}.
\end{eqnarray}
\section{Hausdorff-Young-Paley inequalities}
\label{SEC:Hardy_Littlewood_Paley_inequalities}

A Paley-type inequality for the group Fourier transform on commutative compact quantum group $\G=C(G)$ has been obtained in
 \cite{ANR2015}.
Here we give an analogue of this inequality on arbitrary compact quantum group $\G$.
\begin{thm}[Paley-type inequality]
\label{THM:Paley_inequality} 
Let $1<p\leq 2$ and let $\G$ be a compact quantum group. If  $\varphi(\pi)$ 
is a positive sequence over $\Gh$
 such that
	 \begin{equation}
 \label{EQ:weak_symbol_estimate}
	 M_\varphi:=\sup_{t>0}t\sum\limits_{\substack{\pi\in\Gh\\ \varphi(\pi)\geq t }}
	 \dpi\npi
	 <\infty
	 \end{equation}
is finite, then we have 
\begin{equation}
\label{EQ:Paley_inequality}
\left(
\sum\limits_{\pi\in\Gh}
\dpi\npi
\left(
\frac{\|\widehat{f}(\pi)\|_{\HS}}{\sqrt{\npi}}
\right)^p
{\varphi(\pi)}^{2-p}
\right)^{\frac1p}
\lesssim
M_\varphi^{\frac{2-p}p}
\|f\|_{L^p(\G)}.
\end{equation}
\end{thm}

\begin{thm}[Hausdorff-Young-Paley inequality]
	\label{Cor:general_Paley_inequality}
	Let $1<p\leq b \leq p'<\infty$ and let $\G$ be a compact quantum group.
	If a positive sequence $\varphi(\pi)$, $\pi\in\Gh$, satisfies condition 
	 \begin{equation}
 \label{EQ:weak_symbol_estimate2}
	 M_\varphi:=\sup_{t>0}
	 t\sum\limits_{\substack{\pi\in\Gh\\ \varphi(\pi)\geq t }}\dpi\npi<\infty,
	 \end{equation}
then we have
	\begin{equation}
	\label{EQ:general_Paley_inequality}
	\left(
	\sum\limits_{\pi\in\Gh}
\dpi\npi
	\left(\frac{\|\widehat{f}(\pi)\|_{\HS}}{\sqrt{\npi}}
	{\varphi(\pi)}^{\frac1b-\frac{1}{p'}}
	\right)^b
	\right)^{\frac1b}
	\lesssim
	M_\varphi^{\frac1b-\frac1{p'}}
	\|f\|_{L^p(\G)}.
	\end{equation}
\end{thm}
Further, we recall a result on the interpolation of weighted spaces from \cite{BL2011}:
\begin{thm}[Interpolation of weighted spaces]
\label{THM:L_p-weighted-interpolation}
 Let us write $d\mu_0(x)=\omega_0(x)d\mu(x),\\ d\mu_1(x)=\omega_1(x)d\mu(x)$, and write $L^p(\omega)=L^p(\omega d\mu)$ for the weight $\omega$. \\Suppose that $0<p_0,p_1<\infty$. Then 
$$
	(L^{p_0}(\omega_0), L^{p_1}(\omega_1))_{\theta,p}=L^p(\omega),
$$
where $0<\theta<1,\frac1p=\frac{1-\theta}{p_0}+\frac{\theta}{p_1}$, and $\omega=w^{p\frac{1-\theta}{p_0}}_0w^{p\frac{\theta}{p_1}}_1$.
\end{thm}

From this, interpolating between the Paley-type inequality \eqref{EQ:Paley_inequality} in Theorem \ref{THM:Paley_inequality} and Hausdorff-Young inequality \eqref{EQ:H-Y}, we obtain Theorem \ref{Cor:general_Paley_inequality}. Hence, we concentrate on proving Theorem \ref{THM:Paley_inequality}. 
The proof of Theorem \ref{THM:Paley_inequality} is an adaption of the techniques used in \cite{ANR2015}.
\begin{proof}[Proof of Theorem \ref{THM:Paley_inequality}]
Let $\mu$ give measure $\varphi^2(\pi)\dpi\npi,\pi\in\Gh$ to the set consisting of the single point $\{\pi\}, \pi\in\G$, 
and measure zero to a set which does not contain any of these points, i.e.
$$
	\mu\{\pi\}:=\varphi^2(\pi)\dpi\npi.
$$
We define the space $L^p(\G,\mu)$, $1\leq p<\infty$, 
as the space of complex (or real) sequences
$a=\{a_{l}\}_{l\in\widehat{\G}}$ such that
\begin{equation}\label{EQ:Lpmu}
\|a\|_{L^p(\G,\mu)}
:=
\left(
\sum\limits_{\substack{l\in\widehat{\G}}}
|a_{l}|^p
\varphi^2(\pi)\dpi\npi
\right)^{\frac1p}
<\infty.
\end{equation}
We will show that the sub-linear operator
$$
	T\colon L^p(\G)\ni f \mapsto Tf=a=
	\left\{\frac{\|\widehat{a}(\pi)\|_{\HS}}{\sqrt{\npi}\varphi(\pi)}\right\}_{\pi\in\Gh}\in L^p(\Gh,\mu)
$$
is well-defined and bounded from $L^p(\G)$ to $L^p(\Gh,\mu)$ for $1<p\leq 2$. 
In other words, we claim that we have the estimate
$$
\label{Paley_inequality_alt}
	\|Tf\|_{L^p(\Gh,\mu)}
=
\left(
\sum\limits_{\substack{\pi\in\G}}
\left(
\frac{\|\widehat{f}(\pi)\|_{\HS}}{\sqrt{\npi}\varphi(\pi)}
\right)^p
\varphi^2(\pi)
\dpi\npi
\right)^{\frac1p}
\lesssim
K^{\frac{2-p}{p}}_{\varphi}
\|f\|_{L^p(\G)},
$$
which would give \eqref{EQ:Paley_inequality}
and where we set 
$K_{\varphi}:=\sup_{s>0}s\sum\limits_{\substack{\pi\in\G\\ \varphi(\pi)\geq s}}\dpi\npi$. 
We will show that $A$ is of weak type (2,2) and of weak-type (1,1). For definition and discussions we refer to \cite{Youn2017} where a suitable version of the Marcinkiewicz interpolation theorem is formulated.
More precisely, with the distribution function 
$$
\nu_{\G}(Q)
=
\sum\limits_{\pi\in Q}\dpi\npi
$$
we show that
\begin{eqnarray}
\label{EQ:THM:Paley_inequality_weak_1}
\nu_{\G}(y;Tf)
&\leq &
\left(\frac{M_2\|f\|_{L^2(\G)}}{y}\right)^2  \quad\text{with norm } M_2 = 1,\\
\label{EQ:THM:Paley_inequality_weak_2}
\nu_{\G}(y;Tf)
&\leq &
\frac{M_1\|f\|_{L^1(\G)}}{y}  \qquad\text{with norm } M_1 = K_{\varphi}.
\end{eqnarray}
Then \eqref{EQ:Paley_inequality} would follow by Marcinkiewicz interpolation 
Theorem \cite{ANR2016}.
Now, to show \eqref{EQ:THM:Paley_inequality_weak_1}, using Plancherel's identity \eqref{EQ:plancherel-1},
 we get
\begin{align*}
		y^2
		\nu_{\G}(y;Tf)
		\leq
	\|Tf\|^2_{L^2(\G,\mu)}
&=
\sum\limits_{\substack{\pi\in\widehat{\G}}}
\left(
\frac{\|\widehat{f}(\pi)\|_{\HS}}{\sqrt{\npi}\varphi(\pi)}
\right)^2
\varphi^2(\pi)
\dpi\npi
\\ &=
\sum\limits_{\substack{\pi\in\widehat{\G}}}
\dpi
\|\widehat{f}(\pi)\|^2_{\HS}
=
\|\widehat{f}\|^2_{\ell^2(\widehat{\G})}
=
\|f\|^2_{L^2(\G)}.
\end{align*}
Thus, $T$ is of type (2,2) with norm $M_2\leq1$.
Further, we show that $T$ is of weak-type (1,1) with norm $M_1=C$; more precisely, we show that
\begin{equation}
\label{weak_type}
\nu_{\G}\{\pi\in\widehat{\G} \colon \frac{\|\widehat{f}(\pi)\|_{\HS}}{\sqrt{\npi}\varphi(\pi)} > y\}
\lesssim
K_{\varphi}\,
\dfrac{\|f\|_{L^1(\G)}}{y}.
\end{equation}
The left-hand side here is the weighted sum $\sum \varphi^2(\pi)\dpi\npi$ taken over those $\pi\in\Gh$ for which $\dfrac{\|\widehat{f}(\pi)\|_{\HS}}{\sqrt{\npi}\varphi(\pi)}>y$. 

From definition of the Fourier transform it follows that
$$
	\|\widehat{f}(\pi)\|_{\HS}\leq \sqrt{\npi}\|f\|_{L^1(\G)}.
$$
Therefore, we have
$$
y<\frac{\|\widehat{f}(\pi)\|_{\HS}}{\sqrt{\npi}\varphi(\pi)}
\leq
\frac{\|f\|_{L^1(\G)}}{\varphi(\pi)}.
$$
Using this, we get
$$
\left\{
\pi\in\Gh
\colon 
\frac{\|\widehat{f}(\pi)\|_{\HS}}{\sqrt{\npi}\varphi(\pi)}>y
\right\}
\subset
\left\{
\pi\in\Gh
\colon 
\frac{\|f\|_{L^1(\G)}}{\varphi(\pi)}>y
\right\}
$$
for any $y>0$. Consequently,
$$
\mu\left\{
\pi\in\Gh
\colon 
\frac{\|\widehat{f}(\pi)\|_{\HS}}{\sqrt{\npi}\varphi(\pi)}>y
\right\}
\leq
\mu\left\{
\pi\in\Gh
\colon
\frac{\|f\|_{L^1(\G)}}{\varphi(\pi)}>y
\right\}.
$$
Setting $v:=\frac{\|f\|_{L^1(\G)}}{y}$, we get

\begin{equation}
\label{PI_intermed_est_1}
\mu\left\{
\pi\in\Gh
\colon 
\frac{\|\widehat{f}(\pi)\|_{\HS}}{\sqrt{\npi}\varphi(\pi)}>y
\right\}
\leq
\sum\limits_{\substack{\pi\in\Gh \\ \varphi(\pi)\leq v}}
\varphi^2(\pi)
\dpi\npi.
\end{equation}
We claim that
\begin{equation}\label{EQ:aux1}
\sum\limits_{\substack{\pi\in\Gh \\ \varphi(\pi)\leq v}}
\varphi^2(\pi)
\dpi\npi
\lesssim 
K_{\varphi}
v.
\end{equation}
In fact, we have
$$
	\sum\limits_{\substack{\pi\in\Gh\\ \varphi(\pi)\leq v}}
\varphi^2(\pi)
\dpi\npi
=
	\sum\limits_{\substack{\pi\in\Gh \\ \varphi(\pi)\leq v}}
\dpi\npi
\int\limits^{\varphi^2(\pi)}_0 d\tau.
$$
We can interchange sum and integration to get
$$
	\sum\limits_{\substack{\pi\in\Gh \\ \varphi(\pi)\leq v}}
\dpi\npi
\int\limits^{\varphi^2(\pi)}_0 d\tau
=
\int\limits^{v^2}_0 d\tau \sum
\limits_{\substack{\pi\in\widehat{\G} \\ \tau^{\frac12}\leq\varphi(\pi)\leq v}}
\dpi\npi
.
$$
Further, we make a substitution $\tau=s^2$, yielding
$$
\int\limits^{v^2}_0 d\tau \sum\limits_{\substack{\pi\in\Gh \\ 
\tau^{\frac12}\leq\varphi(\pi)\leq v}}
\dpi\npi
=
2\int\limits^{v}_0 s\,ds 
\sum\limits_{\substack{\pi\in\widehat{\G} \\ s \leq \varphi(\pi)\leq v}}
\dpi\npi
\leq
2\int\limits^{v}_0 s\,ds 
\sum\limits_{\substack{\pi\in\widehat{\G} \\ s \leq\varphi(\pi)}}
\dpi\npi.
$$
Since 
$$
	s\sum\limits_{\substack{\pi\in\Gh \\ s \leq \varphi(\pi) } } d^2_{\pi}
	\leq 
\sup_{s>0}	s\sum\limits_{\substack{\pi\in\Gh \\ s \leq\varphi(\pi) } } 
\dpi\npi
=:K_{\varphi}
$$
is finite by the definition of $K_{\varphi}$, we have
$$
	2\int\limits^{v}_0 s\,ds 
\sum\limits_{\substack{\pi\in\Gh \\ s \leq\varphi(\pi)}}
\dpi\npi
\lesssim K_{\varphi} v.
$$
This proves \eqref{EQ:aux1}.
We have just proved inequalities
\eqref{EQ:THM:Paley_inequality_weak_1},
\eqref{EQ:THM:Paley_inequality_weak_2}.
Then by using \\Marcinkiewicz' interpolation theorem (see Remark \ref{REM:interpolation} below)  with $p=1, q=2$ and 
$\frac1p=1-\theta+\frac{\theta}2$ we now obtain
$$
\left(
\sum\limits_{\substack{\pi\in\Gh}}
\left(
\frac{\|\widehat{f}(\pi)\|_{\HS}}{\sqrt{\npi}\varphi(\pi)}
\right)^p
\varphi^2(\pi)
\dpi\npi
\right)^{\frac1p}
=
\|Tf\|_{L^p(\Gh)}
\lesssim
K^{\frac{2-p}{p}}_{\varphi}
\|f\|_{L^p(\G)}.
$$
This completes the proof.

\end{proof}
\begin{rem}
\label{REM:interpolation}
Let $M$ be a von Neumann algebra with a distinguished faithful normal state $\varphi_{0}$. Haagerup constructed $L^p$-spaces on general von Neumann algebra.
In \cite{Kosaki1984b} Kosaki showed that the Haagerup $L^p$-spaces are interpolation spaces. Let $L^1(M)$ be the predual of $M$. The following holds true \cite{Kosaki1984b}
$$
(L^1(M),M)_{\theta}=L^p(M),\quad \frac1p=1-\theta,\, 0<\theta<1.
$$	
Hence, one immediately obtains Marcinkiewicz interpolation theorem for linear mappings between $L^1(M)\cap M$ and the space $\Sigma$ of matrix valued sequences. We refer to \cite{ANR2016} for more details. One can easily adapt \cite[Theorem 6.1]{ANR2016} to the setting of compact quantum groups.
\end{rem}
\label{techniques}

\section{Fourier multipliers on compact quantum groups}

\begin{defn} Let $(\G,\Delta)$ be a compact quantum group. A linear operator $A\colon \G\to \G$ is called a left Fourier multiplier if 
\begin{equation}
\Delta\circ A=(\Id\otimes A)\circ \Delta.
\end{equation}
\end{defn}
\begin{ex} Let $\G=C(G)$ where $G$ is a compact topological group. Then $A$ is a left Fourier multiplier if and only if
$$
A\circ T_g=T_g\circ A,
$$
where $T_g\colon f(\cdot)\to f(g\cdot)$ is translation. It can then be shown 
(see for example \cite{RuzhanskyTurunen10})
 that
$$
\widehat{Af}(\pi)=\sigma_A(\pi)\widehat{f}(\pi),
$$
where $\sigma_{A}(\pi)\in\C^{n_{\pi}\times n_{\pi}}$ is the global symbol of $A$.
\end{ex}
\begin{thm} 
\label{THM:co-invariant-FM}
Let $\G$ be a compact quantum group and let $A\colon \G\to \G$ be a left Fourier multiplier. Then 
\begin{equation}
\label{EQ:multiplier}
\widehat{Af}(\pi)
=
\sigma_A(\pi)
\widehat{f}(\pi),\quad f\in L^2(\G),
\end{equation}
where 
$\sigma_A(\pi)\in\mathbb{C}^{\dpi\times \dpi}$
are defined by $A\pi=\pi\sigma_A(\pi)$.
\end{thm}
\begin{proof}[Proof of Theorem \ref{THM:co-invariant-FM}]
By the Peter-Weyl decomposition \eqref{EQ:Peter-Weyl-decomposition}, it is sufficient to establish \eqref{EQ:multiplier}  on the coefficient sub-colagebra $\mathcal{C}(\pi)$.
Suppose $A$ is left invariant and 
write 
$$
A\pi_{ij}
=
\sum\limits_{\pi'\in\Gh}
\sum\limits^{n_{\pi}}_{i,k=1}
\pi'_{kl}
c^{\pi'}_{ijkl}
$$ 
for some coefficients $c^{\pi'}_{ijkl}$. Then by left invariance,
\[ \Delta A \pi_{ij}=\sum_{\pi',k,l,m}c^{\pi'}_{ijkl} \pi'_{km}\tens\pi'_{ml}= (\id\tens A)\Delta \pi_{ij}=\sum_m \pi_{im}\tens A\pi_{mj}
=
\sum\limits_m\sum\limits_{l,k}C^{\pi'}_{mjkl}\pi_{im}\otimes\pi'_{kl}
\]
Comparing these, by the Peter-Weyl decomposition, we see that only $\pi'=\pi$ can contribute. Moreover, since the $\{\pi_{ij}\}$ are a basis of $C(\pi)$ we must have $c^{\pi}_{ijkl}=0$ unless $k=i$. So we have $\sum_{l,m}c^{\pi}_{ijil} \pi_{im}\tens\pi_{ml}=\sum_m\pi_{im}\tens A\pi_{mj}$. Comparing these we see that 
\[ A\pi_{mj}=\sum_l \pi_{ml}\sigma_A(\pi)_{lj}\]
 from some matrix $\sigma_A(\pi)_{lj}=c^\pi_{ijil}$ which can not depend on $i$. Finally, setting $f=\pi_{kl}$ we have
 \[ \widehat{\pi_{kl}}(\pi')_{ij}=h(\pi_{kl}\pi'{}^*_{ji})=\delta_{\pi,\pi'}{q^\pi_k\over d_\pi}\delta_{kj}\delta_{li},\]
 and we check that 
 \[ \widehat{A\pi_{kl}}(\pi')_{ij}=\sum_m \widehat{\pi_{km}}(\pi')_{ij}\sigma_A(\pi)_{ml} =\delta_{\pi,\pi'}{q^\pi_k\over d_\pi} \delta_{kj}\sigma_A(\pi)_{il},\]
 \[ (\sigma_A(\pi') \widehat{\pi_{kl}}(\pi'))_{ij}=\sum_m\sigma_A(\pi')_{im} \widehat{\pi_{kl}}(\pi')_{mj}=\sigma_A(\pi)_{il}\delta_{\pi,\pi'}{q^\pi_k\over d_\pi}\delta_{kj},\]
 which is the same. This proves Theorem 4.3. 
 \end{proof}
Essentially similar arguments can be found in an earlier paper by \cite{Cipriani2014a}.
Also note from the proof that the same result applies to any coinvariant linear map $A\colon \C[\G]\to \C[\G]$. We refer to the operators $A$ acting in this way on the Fourier side as quantum  Fourier multipliers. 
In the classical situation on $\G=\TT^n$, left Fourier multipliers are essentially operators acting via convolution with measures whose Fourier coefficients are bounded.

Let $A\colon \G\to \G $ be a left Fourier multiplier.
We are concerned with the question of what assumptions on the symbol $\sigma_A$
guarantee that $A$ is bounded from $L^p(\G)$ to $L^q(\G)$.
\begin{thm}\label{THM:upper}
Let $1<p\leq 2\leq q<\infty$ and let $A\colon L^2(\G)\to L^2(\G)$ be a left Fourier multiplier.
Then we have
\begin{equation}
\label{EQ:upper}
	\|A\|_{L^{p}(\G)\to L^{q}(\G)}
\lesssim
\sup_{s>0}s\left( \sum_{\substack{\pi\in\Gh \\ \|\sigma_A(\pi)\|_{\op}>s}} \dpi\npi\right)^{\frac1p-\frac1{q}}.	
\end{equation}
\end{thm}
\begin{proof}[Proof of Theorem \ref{THM:upper}]
By definition
\begin{equation}
\widehat{Af}(\pi)=\sigma_A(\pi)\widehat{f}(\pi).
\end{equation}
Let us first assume that $p\leq q'$.
Since $q'\leq 2$, for $f\in \C[\G]$ the Hausdorff-Young inequality gives 
\begin{align}
\label{EQ:Af-norm}
\begin{split}
\|Af\|_{L^{q}(\G)}
&
\leq
\|\widehat{Af}\|_{\ell^{q'}(\Gh)}
=
\|\sigma_A\widehat{f}\|_{\ell^{q'}(\Gh)}
=
\left(
\sum\limits_{\pi\in\Gh}
\dpi\npi
\left(
\frac{
\|\sigma_A(\pi)\widehat{f}(\pi)\|_{\HS}
}
{\sqrt{\npi}}
\right)^{q'}
\right)^{\frac1{q'}}
\\&\leq
\left(
\sum\limits_{\pi\in\G}
\dpi\npi
\|\sigma_A(\pi)\|^{q'}_{\op}\left(\frac{\|\widehat{f}(\pi)\|_{\HS}}{\sqrt{\npi}}\right)^{q'}
\right)^{\frac1{q'}}
.
\end{split}
\end{align}
The case $q'\leq (p')'$ can be reduced to the case $p\leq q'$ as follows.
The $L^p$-duality yields
\begin{equation}
\label{EQ:A-A-star-norm}
\|A\|_{L^{p}(\G)\to L^{q}(\G)}
=
\|A^*\|_{L^{q'}(\G)\to L^{p'}(\G)}.
\end{equation}
The symbol $\sigma_{A^*}(\pi)$ of the adjoint operator $A^*$ equals to $\sigma_{A}^*(\pi)$,
\begin{equation}
\label{EQ:symbol-conjugate}
\sigma_{A^*}(\pi)=\sigma_A^*(\pi),\quad \pi\in\widehat{\G},
\end{equation}
and its operator norm $\|\sigma_{A^*}(\pi)\|_{\op}$ equals to $\|\sigma_{A}(\pi)\|_{\op}$.
Now, we are in a position to apply Corollary \ref{Cor:general_Paley_inequality}. 
Set $\frac1r=\frac1p-\frac1q$. We observe that
with $\sigma(\pi):=\|\sigma_A(\pi)\|^r_{\op} I_{\dpi},\pi\in\widehat{\G},$ and $b=q'$, the assumptions of Corollary \ref{Cor:general_Paley_inequality} are satisfied and  we obtain
\begin{align*}
\left(
\sum\limits_{\pi\in\Gh}
\dpi\npi
\|\sigma_A(\pi)\|^{q'}_{\op}
\left(
\frac{\|\widehat{f}(\pi)\|_{\HS}}{\sqrt{\npi}}
\right)^{q'}
\right)^{\frac1{q'}}
\lesssim
\left(
\sup_{s>0}s\sum\limits_{\substack{\pi\in\Gh \\ \|\sigma(\pi)\|^r_{\op}>s}}
\dpi\npi
\right)^{\frac1r}
\|f\|_{L^{p}(\G)}
\end{align*}
for all $f\in L^{p}(\G)$, in view of $\frac1{q'}-\frac1{p'}=\frac1p-\frac1q=\frac1r$.
Thus, for $1<p\leq 2 \leq q <\infty$, we obtain
\begin{equation}
\|Af\|_{L^q(\G)}
\lesssim
\left(
\sup_{s>0}s\sum\limits_{\substack{\pi\in\Gh\\\|\sigma(\pi)\|^r_{\op}>s}}
\dpi\npi
\right)^{\frac1r}
\|f\|_{L^p(\G)}.
\end{equation}
Further, it can be easily checked that
\begin{multline*}
\left(
\sup_{s>0}s\sum\limits_{\substack{\pi\in\Gh\\\|\sigma(\pi)\|_{\op}>s}}
\dpi\npi
\right)^{\frac1r}
=
\left(
\sup_{s>0}s\sum\limits_{\substack{\pi\in\Gh\\\|\sigma_A(\pi)\|_{\op}>s^{\frac1r}}}
\dpi\npi
\right)^{\frac1r}
\\=
\left(
\sup_{s>0}s^{r}\sum\limits_{\substack{\pi\in\Gh\\\|\sigma_A(\pi)\|_{\op}>s}}
\dpi\npi
\right)^{\frac1r}
=
\sup_{s>0}s \left(\sum\limits_{\substack{\pi\in\Gh\\\|\sigma_A(\pi)\|_{\op}>s}}
\dpi\npi
\right)^{\frac1r}.
\end{multline*}
This completes the proof.
\end{proof}

\section{Hardy-Littlewood inequality and spectral triples}
As a corollary of Theorem \ref{THM:Paley_inequality}, we obtain a formal compact quantum group version of the Hardy-Littlewood inequality by using a suitable sequence $\{\lambda_{\pi}\}$.

\begin{thm}
\label{THM:HL-SUQ2}
Let $1<p\leq 2$ and let $\G$ be a compact quantum group. Assume that a sequence $\{\lambda_{\pi}\}_{\pi\in\Gh}$ grows sufficiently fast, that is,
\begin{equation}
\label{EQ:series-conv}
\sum\limits_{\pi\in\Gh}\frac{\dpi\npi}{\left|\lambda_{\pi}\right|^{\beta}}<\infty.
\end{equation}
Then we have
\begin{equation}
\label{EQ:HL}
\sum\limits_{\pi\in\Gh}
\dpi\npi
\left|\lambda_{\pi}\right|^{\beta(p-2)}
\left(
\frac{\|\widehat{f}(\pi)\|_{\HS}}{\sqrt{n_{\pi}}}
\right)
^p
\lesssim
\|f\|_{L^p(\G)}.
\end{equation}
\end{thm}
The word `formal' stands for the fact that we do not study underlying inherent geometric data of the group. Assuming the existence of the sequence $\{\lambda_{\pi}\}$ with condition \eqref{EQ:series-conv} does not provide us with geometric condition. Nevertheless, we show later in Theorem \ref{THM:D-a-bounded} that one can indeed obtain non-trivial family of spectral triples.

\begin{proof}[Proof of Theorem \ref{THM:HL-SUQ2}]
By the construction 
\begin{equation}
\label{EQ:key-condition}
C:=\sum\limits_{\pi\in\Gh}\frac{\dpi\npi}{\left|\lambda_{\pi}\right|^{\beta}}<+\infty.
\end{equation}
Then we have
\begin{equation*}
C\geq
\sum\limits_{\substack{\pi\in\Gh\\ \left|\lambda_{\pi}\right|^{\beta}\leq \frac1t}}
\frac{\dpi\npi}{\left|\lambda_{\pi}\right|^{\beta}}
\geq
t
\sum\limits_{\substack{\pi\in\Gh\\ \left|\lambda_{\pi}\right|^{\beta}\leq \frac1t}}
\dpi\npi
=
t
\sum\limits_{\substack{\pi\in\Gh\\ \frac1{\left|\lambda_{\pi}\right|^{\beta}}\geq t}}
\dpi\npi
.
\end{equation*}
Then by Theorem \ref{THM:Paley_inequality}, we get
$$
\sum\limits_{\pi\in\Gh}
\dpi\npi
\left|\lambda_{\pi}\right|^{\beta(p-2)}\left(\frac{\|\widehat{f}(\pi)\|_{\HS}}{\sqrt{\npi}}\right)^p
\lesssim
\|f\|_{L^p(\G)}.
$$
This completes the proof.

\end{proof}

\begin{defn}
\label{DEF:spectral-triple}
 A 'bare' spectral triple $(\A,\H,\D)$ is a triple consisting of an associative $*$-subalgebra $\A$ of the algebra $B(\H)$ of bounded operators in a separable Hilbert space $\H$ and a linear closed unbounded operator $\D\colon \H\to\H$ with discrete spectrum $\{\lambda_k\}$ and the polar decomposition $\D=\sign(\D)\left|\D\right|$ such that
\begin{equation}
\label{EQ:bracket-bounded}
\A\ni a\mapsto \partial(a):=[\left|\D\right|,\pi(a)]\in B(\H),
\end{equation}
where $\pi\colon \A\to B(\H)$ is a $*$-representation of $\A$ on $B(\H)$.
A spectral triple $(\A,\H,\D)$ is called summable if there is $\beta\in\RR_+$ such that
\begin{equation}
\label{EQ:summability}
\sum\limits^{\infty}_{k=1}\frac1{\left|\lambda_k\right|^{\beta}}<+\infty,\quad \text{for some $\beta>0$}.
\end{equation}
The minimal $\beta\in\RR_+$ such that \eqref{EQ:summability} holds is called the spectral dimension \cite{Connes1996}.
\end{defn}

Definition \ref{DEF:spectral-triple} is very minimal in the sense that we do not impose any conditions on reality and chiral operators and their interrelations with $\D$ \cite{Connes1995}.

Then we show in Theorem \ref{THM:D-a-bounded} that $\D$ defined by \eqref{EQ:Dirac} yields a spectral triple in the sense of Definition \ref{DEF:spectral-triple}.

\begin{defn}[Smooth domain] 
\label{DEF:smooth}
Let $\G$ be a compact quantum group and let $\D\colon \C[\G]\to \C[\G]$ be a linear map extended to $L^2(\G)\to L^2(\G)$ as a closed unbounded linear operator.
Then the smooth domain $C^{\infty}_{\D}\subset\G$ of $\D$ is defined as follows
\begin{equation}
C^{\infty}_{\D}:=\bigcap\limits_{\alpha\geq 0}\Dom(\left|\D\right|^{\alpha}).
\end{equation}
The Frechet structure is given by the seminorms
\begin{equation}
\|\varphi\|_{\alpha}=\|\left|\D\right|^{\alpha}\varphi\|_{\H},\quad\varphi\in C^{\infty}_{\D}, \alpha\geq 0.
\end{equation}
\end{defn}
The powers $\left|\D\right|^{\alpha}$ are defined by the spectral theorem.
It can be checked that $C^{\infty}_{\D}$ is a locally convex topological vector space.

Let $\pi^k,\pi^s\in\Gh$. Then the tensor product $\pi^k\otimes\pi^s$ is a completely reducible finite-dimensional representation. The matrix elements of $\pi^k\otimes\pi^s$ are given by
$$
\pi^k\otimes\pi^s
=
[
\pi^k_{ij}\pi^s_{pr}
]^{n_k, n_s}_{i,j=1, p,r=1}.
$$
We shall define the  coefficients $C^{ksm}_{ijprut}$ as follows
\begin{equation}
\label{EQ:c-g-coeff-0}
C^{ksm}_{ijprut}
=
(\pi^k_{ij}\pi^s_{pr},\pi^m_{ut})_{L^2(\G)}.
\end{equation}
It then follows from \eqref{EQ:c-g-coeff-0}
\begin{equation}
\label{EQ:c-g-coeff}
\pi^k_{ij}\pi^s_{pr}
=
\bigoplus\limits_{m \in I_{ks}}\sum\limits^{n_m}_{u,t=1}
C^{ksm}_{ijprut}\pi^m_{ut},
\end{equation}
where $I_{ks}$ is a finite subset of $\NN$. 

The Clebsch-Gordan coefficients are important to write down the action of the commutator $\partial(a)=[\left|\D\right|,a]$ explicitly. In \cite{Chakraborty2008}, these coefficients were computed for the quantum groups ${\mathrm{SU}_q(2l+1)}$. This allowed to write down explicitly the action of the left multiplication operator on $L^2({\mathrm{SU}_q(2l+1)})$ leading in turn to the growth restriction on the eigenvalues $\lambda_k$. In this paper, we take a different approach.
In comparison with \cite{Chakraborty2008}, we cannot compute our version of Clebsch-Gordan coefficients $C^{ksm}_{ijprut}$ explicitly. Nevertheless, we can build a "nice" subalgebra $C^{\infty}_{\D}(\G)$ of $\G$ which is bounded under the commutation with a Dirac operator $\D$.
\begin{thm}
\label{THM:D-a-bounded}
Let $(\G,\Delta)$ be a compact quantum group and let $\D\colon \C[\G]\to \C[\G]$ be a linear operator given by
\begin{equation}
\label{EQ:Dirac}
\D\pi^k_{ij}=\lambda_{k}\pi^k_{ij},
\end{equation}
where $\{\lambda_{k}\}$ is a sequence of real numbers such that
\begin{equation}
\label{EQ:beta-dimension}
\sum\limits_{k\in\NN}\frac{d_kn_k}{\left|\lambda_{k}\right|^{\beta}}<+\infty
\end{equation}
for some $\beta>0$.
Assume that for all $k,s$ 
\begin{equation}
\label{PROP:EQ:D-a-bounded}
\left|\lambda_{k}-\lambda_{s}\right|
\sqrt
{
\sum\limits_{m\in I_{ks}}
\sum\limits^{n_m}_{t=1}
\left|
C^{ksm}_{ijprtt}
\right|^2
\frac{q^{m}_t}{d_{\pi^m}}
}
\leq
C
\sqrt{
\frac{q^{s}_r}{d_{\pi^s}}
}.
\end{equation}

Then the subspace $C^{\infty}_{\D}(\G)$ is a bare spectral triple. Moreover, the commutator 
\begin{equation}
\partial(a)\colon \H\ni b\mapsto \partial(a)b=[\left|\D\right|a,b]\in \H
\end{equation}
is bounded if and only if conditions \eqref{PROP:EQ:D-a-bounded} is satisfied.
\end{thm}
\begin{ex}[Equivariant spectral triples on the quantum $\mathrm{SU}_q(n)$] 
Condition \eqref{PROP:EQ:D-a-bounded} imposes certain growth condition on consecutive differences $\left|\lambda_k-\lambda_{s(k)}\right|$ of the eigenvalues $\lambda_k$.
For $\G=\mathrm{SU}_q(n)$, it is possible to compute \cite{Chakraborty2008} the coefficients $C^{ksm}_{ijprut}$. In more detail, the authors showed that the $C^{ksm}_{ijprtt}$'s are essentially powers of $q$, i.e.
\begin{equation}
C^{ksm}_{ijprtt}
=
q^{C},
\end{equation}
where the exponent $C'=C'(k,s,m,i,k,p,r,t)$ is determined by $k,s,m$. For more details we refer to \cite[pp. 30-32]{Chakraborty2008}.
\end{ex}
\begin{proof}[Proof of Theorem \ref{THM:D-a-bounded}]
We give the proof in two steps.

{Step 1: Closed under the multiplication.}

By Definition \ref{DEF:smooth}, it is sufficient to show that for every $a,b\in C^{\infty}_{\D}$, we have
\begin{equation}
\label{EQ:C_D_algebra_step_1}
\|\left|\D\right|^{\alpha}(ab)\|_{\H}<+\infty,\quad \alpha\in\RR_+.
\end{equation}
By interpolation, it is sufficient to establish \eqref{EQ:C_D_algebra_step_1} for integer $\alpha=k\in\NN_0$.
We denote
\begin{equation}
\partial(a)\colon \G\mapsto \partial(a)(b)=[\left|\D\right|,a]b=\left|\D\right|ab-a\left|\D\right|b.
\end{equation}
For $k=1$, it is straightforward to check that
\begin{align*}
&\|\left|\D\right| ab-a\left|\D\right| b+a\left|\D\right| b\|
=
\|[\left|\D\right|,a]b+a\left|\D\right| b\| \\
&\leq
\|[\left|\D\right|a]b\|_{\H}
+
\|a \left|\D\right|b\|_{\H}
\leq
C_{a,\D}\|b\|_{\H}+p_1(a)\| b\|_{\H}
=
(C_{a,\D}+p_1(a))\| b\|_{\H}
<+\infty.
\end{align*}
For $k=2$, we have
\begin{equation}
\left|\D\right|^2ab
=
\partial^2(a)b+\partial(a)(\left|\D\right| b)+a\left|\D\right|^2b.
\end{equation}
By mathematical induction, we establish \eqref{EQ:C_D_algebra_step_1} for $\alpha=k\in\NN$. 

{Step 2: The commutators are bounded. }
First, we show necessity. Assume that $\partial(a)$ is bounded for all $a\in C^{\infty}_{\D}$, i.e.
\begin{equation}
\|\partial(a)b\|_{\H}
\leq C_a
\|b\|_{\H},\quad b\in C^{\infty}_{\D}.
\end{equation}
Let $\pi^k,\pi^s\in\Gh$ and take $a=\pi^k_{ij}$ and $b=\pi^s_{pq}$. Then by the direct computation
\begin{align}
\begin{split}
\label{EQ:D-pi}
\left\|
\partial(\pi^k_{ij})\pi^s_{pq}
\right\|^2_{L^2(\G)}
&=
\|
(\D\pi^k_{ij})\pi^s_{pq}-\pi^k_{ij}(\D\pi^s_{pq})
\|^2_{L^2(\G)}
=
\|
\lambda_{k}\pi^k_{ij}\pi^s_{pq}-\lambda_{s}\pi^k_{ij}\pi^s_{pq}
\|^2_{L^2(\G)}
\\ &=
\left|
\lambda_{k}-\lambda_{s}
\right|^2
\|\pi^k_{ij}\pi^s_{pq}\|^2_{L^2(\G)}
\\ &=
\left|\lambda_{k}-\lambda_{s}\right|^2
\sum\limits_{m\in I_{ks}}
\sum\limits^{n_{m}}_{t=1}
\left|
C^{ksm}_{ijpqtt}
\right|^2
\frac{q^{m}_t}{d_{\pi^m}}.
\end{split}
\end{align}
Using that 
$$
\|\pi^s_{pq}\|^2_{L^2(\G)}
=
\frac{q^{s}_q}{d_{\pi^s}}
$$
and \eqref{EQ:D-pi}, we obtain \eqref{PROP:EQ:D-a-bounded}.

Now, we show that condition \eqref{PROP:EQ:D-a-bounded} is sufficient for the boundedness of $\partial(a)$.

Let $a,b\in C^{\infty}_{\D}(\G)$. Their Fourier expansions are given by
\begin{eqnarray}
\label{EQ:a-fourier}
a
=
\sum\limits_{k\in\NN}
\sum\limits^{n_k}_{i,j=1}
d_k
\frac1{q^k_i}
\widehat{a}(k)_{ji}\pi^k_{ij},
\\
\label{EQ:b-fourier}
b
=
\sum\limits_{s\in\NN}
\sum\limits^{n_s}_{p,r=1}
d_s
\frac1{q^s_p}
\widehat{b}(s)_{rp}\pi^s_{pr},
\end{eqnarray}
where, for brevity, we denoted
$
\widehat{a}(k)=\widehat{a}(\pi^k),\, \widehat{b}(s)=\widehat{b}(\pi^s).
$
We shall now show that
\begin{equation}
\|\partial(a)b\|^2_{L^2(\G)}
\leq
C_a\|b\|^2_{L^2(\G)}.
\end{equation}
Plugging Fourier expansions \eqref{EQ:a-fourier} and \eqref{EQ:b-fourier}, we can express the commutator 
\begin{equation}
(\D a-a\D)b
=
\sum\limits
d_kd_s
\left(
\lambda_k-\lambda_s
\right)
\frac1{q^k_i}
\widehat{a}(k)_{ji}
\frac1{q^s_p}
\widehat{b}(s)_{rp}
C^{ksm}_{ijprut}
\pi^m_{ut}.
\end{equation}
Hence, we get
\begin{equation}
\label{EQ:step-1-D-a}
\|\partial(a)b\|^2_{L^2(\G)}
=
\sum\limits
\left|
d_kd_s
\left(
\lambda_k-\lambda_s
\right)
\frac1{q^k_i}
\widehat{a}(s)_{ji}
\frac1{q^s_p}
\widehat{b}(k)_{rp}
C^{ksm}_{ijprtt}
\right|^2
\frac{q^{m}_t}{d_m},
\end{equation}
where we used \eqref{EQ:c-g-coeff} and the Peter-Weyl orthogonality relations \eqref{EQ:peter-weyl-orth}.

Applying condition \eqref{PROP:EQ:D-a-bounded}, we get from \eqref{EQ:step-1-D-a}
\begin{equation}
\label{EQ:step-2-D-a}
\|\partial(a)b\|^2_{L^2(\G)}
\leq
C
\sum\limits
\left|
d_kd_s
\frac1{q^k_i}
\widehat{a}(s)_{ji}
\frac1{q^s_p}
\widehat{b}(k)_{rp}
\right|^2
\frac{q^s}{d_s}
=
\sum\limits
d_s
\frac1{q^s_p}
\left|
\widehat{b}(k)_{rp}
\right|^2
\sum\limits
\left(
\frac{d_k}{q^k_i}
\right)^2
\left|
\widehat{a}(s)_{ji}
\right|^2.
\end{equation}
Recalling that
$\sum\limits
d_s
\frac1{q^s_p}
\left|
\widehat{b}(k)_{rp}
\right|^2
=\|b\|^2_{L^2(\G)}
$, we obtain
\begin{equation}
\label{EQ:D-a-estimate}
\|\partial(a)b\|^2_{L^2(\G)}
\leq
\|b\|^2_{L^2(\G)}
\sum\limits
\left(
\frac{d_k}{q^k_i}
\right)^2
\left|
\widehat{a}(s)_{ji}
\right|^2.
\end{equation}
It is clear that
\begin{equation}
\label{EQ:d_k-q_i}
\frac{d_k}{q^k_i}
\leq
d_kn_k
\end{equation}
in the view of $\sum\limits^{n_{k}}_{i=1}q^{k}_i
=
\sum\limits^{n_{k}}_{i=1}\frac1{q^{k}_i}
=\dpi.
$
Here we write
$
d_k=d_{\pi^k},\, n_k=n_{\pi^k}.
$
Since the series $\sum\limits_{k\in\NN}\frac{d_kn_k}{\left|\lambda_k\right|^{\beta}}$ is convergent, we have
\begin{equation}
\label{EQ:d_k-lambda}
d_kn_k\leq C\left|\lambda_k\right|^{\beta}.
\end{equation}
Combining \eqref{EQ:d_k-q_i} and \eqref{EQ:d_k-lambda}, we get
\begin{equation}
\label{EQ:d-lambda-beta}
\frac{d_k}{q^k_i}
\lesssim
\left|\lambda_k\right|^{\beta}.
\end{equation}
Combining \eqref{EQ:D-a-estimate} and \eqref{EQ:d-lambda-beta}, we obtain
\begin{equation}
\|\partial(a)b\|^2_{L^2(\G)}
\leq
\|b\|^2_{L^2(\G)}
\sum\limits
\frac{d_k}{q^k_i}
\left|\lambda_k\right|^{\beta}
\left|
\widehat{a}(k)_{ji}
\right|^2
=
\|b\|^2_{L^2(\G)}
\sum\limits
d_k
\left|\lambda_k\right|^{\beta}
\|\widehat{a}(k)\|^2_{\HS}.
\end{equation}
This completes the proof.
\end{proof}
\begin{ex}\label{exDeltaG} 
Let $\G=C(G)$ where $G$ is a compact Lie group. One can take 
$\D=\sqrt{1-\Delta_{G}}$ where $\Delta_{G}$ is the Laplacian on $G$.
\end{ex}

\begin{ex}[{{\cite{Chakraborty2010}}}]
\label{EX:pal-1}
Let $\mathbb{G}=\SUq2$ and $L^2(\SUq2)$ be the GNS-space. Let $\mathcal{D}$ be a Dirac operator operator acting on the entries $t^l_{ij}$ of the irreducible corepresentations $t^l\in\widehat{\SUq2}$ of $\SUq2$ as follows
\begin{equation}
\mathcal{D}t^{l}_{ij}=\pm(2l+1)t^l_{ij},\quad i,j=0,\frac12,\ldots,l,\,l\in\frac12\NN_0.
\end{equation}
In this example, we have
$
C^{\infty}_{\D}=\SUq2.
$
The Chern character corresponding to $(C^{\infty}_{\D},L^2(\SUq2),\D)$ is non-trivial\cite{Chakraborty2010}.
\end{ex}
We can now formulate quantum Hardy-Littlewood type inequality \eqref{EQ:HL} in a manner similar to the compact Lie group inequality  \eqref{H_L_inequality-alt-G}.
\begin{cor} Let $1<p\leq 2$ and let $(\G,\Delta)$ be a compact quantum group and let $(C^{\infty}_{\D},L^2(\G),\D)$ a $\beta$-summable spectral triple.  Then we have
\begin{equation}
\label{EQ:D-HL}
\| \FT_{\G}\left|\D\right|^{{\beta}(\frac12-\frac1p)}f\|_{\ell^p(\Gh)} \leq
   C_p\|f\|_{L^p(\G)}.
\end{equation}
\end{cor}
\begin{proof}
$$
\left|\D\right|^{\beta(\frac12-\frac1p)}\pi_{ij}
=
\left|\lambda_{\pi}\right|^{\beta(\frac12-\frac1p)}\pi_{ij}.
$$
Using this and the right-hand side in inequality \eqref{EQ:HL}, we obtain \eqref{EQ:D-HL}.
\end{proof}

\section{Schwartz kernels}

Let $(\G,\Delta)$ be a compact quantum group and  let $(C^{\infty}_{\D},L^2(\G),\D)$ be a summable spectral triple. It is clear that $C^{\infty}_{\D}$ is a Frechet space.
We show that every linear operator $A\colon C^{\infty}_{\D}\to C^{\infty}_{\D} $ continuous with respect to the Frechet topology can be associated with the distribution $K_A$ 'acting' on $\G\times \G$. In other words, every linear continuous operator $A\colon C^{\infty}_{\D}\to C^{\infty}_{\D} $ possesses Schwartz kernel $K_A$. This allows us to define global symbol of $A$ in lines with the pseudo-differential calculus on compact Lie groups \cite{Ruzhansky+Turunen-IMRN}, \cite{RuzhanskyTurunen10}. The global symbols have been recently studied \cite{Levy2016} on the quantum tori $\mathbb{T}^{n}_{\theta},n\in\mathbb{N},\theta\in\RR$.

\begin{defn}[Rapidly decreasing functions on $\Gh$]
Denote by $\mathcal{S}(\Gh)$ the space of matrix-valued sequences $\{\sigma(\pi)\},\sigma(\pi)\in\C^{\npi\times\npi}$ satisfying the conditions
\begin{equation}
\mathcal{S}(\Gh)
=
\left\{
\sigma=\{\sigma(\pi)\}_{\pi\in\Gh}
\colon
\sum\limits_{\pi\in\Gh}\dpi\npi
\left|\lambda_{\pi}\right|^{2\alpha}
\|\sigma(\pi)\|^2_{\HS}
<+\infty
\text{ for any $\alpha>0$}
\right\}.
\end{equation}
\end{defn}

The space $\mathcal{S}(\Gh)$ becomes a locally convex topological space if we endow it with the norms
\begin{equation}
\label{EQ:p-norms}
p_{\alpha}(\sigma)
:=
\left(
\sum\limits_{\pi\in\Gh}\dpi\npi
\left|\lambda_{\pi}\right|^{2\alpha}\|\sigma(\pi)\|^2_{\HS}
\right)^{\frac12},\quad \alpha>0.
\end{equation}
There is an equivalent Frechet structure
\begin{equation}
q_{\gamma}(\sigma)
:=
\sup_{\pi\in\Gh}\left|\lambda_{\pi}\right|^{\gamma}\|\sigma(\pi)\|_{\op},\quad \gamma>0.
\end{equation}
\begin{prop} 
\label{PROP:seminorms_equivalence}
Let $\left\{\lambda_{\pi}\right\}$ satisfy condition \eqref{EQ:beta-dimension}. Then two families of seminorms $\{p_{\alpha}\}_{\alpha>0}$ and $\{q_{\gamma}\}_{\gamma>0}$ are equivalent.
\end{prop}
\begin{proof}[Proof of Proposition \ref{PROP:seminorms_equivalence}]
 The two family of seminorms  $\{p_{\alpha}\}_{\alpha>0}$ and $\{q_{\gamma}\}_{\gamma>0}$ are equivalent if
for any $\alpha>0$ there is such $\gamma>0$ that
\begin{equation}
\label{EQ:p-less-q}
p_{\alpha}(\sigma)
\leq
Cq_{\gamma}(\sigma)
\end{equation}
and for any $\gamma$ there is $\alpha$ such that the converse inequality holds
\begin{equation}
\label{EQ:q-less-p}
q_{\gamma}(\sigma)
\leq
Cq_{\alpha}(\sigma).
\end{equation}
First, we show inequality \eqref{EQ:p-less-q} and then \eqref{EQ:q-less-p}.
Using the fact
\begin{equation}
\|\sigma(\pi)\|^2_{\HS}
\leq
\npi\|\sigma(\pi)\|_{\op},
\end{equation}
we estimate
\begin{align*}
p_{\alpha}(\sigma)
&=
\left(
\sum\limits_{\pi\in\Gh}
\dpi\npi
\left|\lambda\right|^{2\alpha}_{\pi}\|\sigma(\pi)\|^2_{\HS}
\right)^{\frac12}\\
& \leq
\left(
\sum\limits_{\pi\in\Gh}
\dpi\npi
\left|\lambda\right|^{2\alpha}_{\pi}\|\sigma(\pi)\|_{\op}
\right)^{\frac12}
=
\left(
\sum\limits_{\pi\in\Gh}
\frac{\dpi\npi}{\left|\lambda_{\pi}\right|^{\beta}}\left|\lambda_{\pi}\right|^{2\alpha+\beta}\|\sigma(\pi)\|^2_{\op}
\right)^{\frac12}\\
& \leq
\left(
\sum\limits_{\pi\in\Gh}\frac{\dpi\npi}{\left|\lambda_{\pi}\right|^{\beta}}
\right)^{\frac12}
\sup_{\pi\in\Gh}\left|\lambda_{\pi}\right|^{\alpha-\frac{\beta}2}\|\sigma(\pi)\|_{\op}
=
\sqrt{\|\D^{-1}\|_{L^{\beta}}}
q_{\alpha-\frac{\beta}2}(\sigma).
\end{align*}
This proves \eqref{EQ:p-less-q} with $\gamma=\alpha-\frac{\beta}2$.
Further, we prove \eqref{EQ:q-less-p}.
It can be shown that
$$
\|\sigma(\pi)\|_{\op}\leq\|\sigma(\pi)\|_{\HS}.
$$
Using this, we get
\begin{equation*}
\left|\lambda_{\pi}\right|^{\gamma}\|\sigma(\pi)\|_{\op}
\leq
\left|\lambda_{\pi}\right|^{\gamma}\|\sigma(\pi)\|_{\HS}
\leq
\left(
\sum\limits_{\xi\in\Gh}d_{\xi}n_{\xi}
\left|\lambda_{\xi}\right|^{\gamma}\|\sigma(\xi)\|^2_{\HS}
\right)^{\frac12}.
\end{equation*}
Taking supremum over all $\pi\in\Gh$, we get
$$
\sup_{\pi\in\Gh}\left|\lambda_{\pi}\right|^{\gamma}\|\sigma(\pi)\|_{\op}
\leq
\left(
\sum\limits_{\xi\in\Gh}d_{\xi}n_{\xi}
\left|\lambda_{\xi}\right|^{\gamma}\|\sigma(\xi)\|^2_{\HS}
\right)^{\frac12}=p_{\gamma}(\sigma)
$$
Hence, we prove \eqref{EQ:q-less-p} with $\gamma=\alpha$.
This completes the proof.
\end{proof}
The construction of the topology on $C^{\infty}_{\D}(\G)$ readily implies 
that the quantum Fourier transform $\FT_{\G}$ is a homeomorphism between $C^{\infty}_{\D}$ and $\mathcal{S}(\Gh)$.
\begin{defn}[Distributions]
Let us denote by $\mathcal{S}'(\G)$ the space  $[C^{\infty}_{\D}(\G)]^*$ of all linear  functionals continuous with respect to the topology on $C^{\infty}_{\D}(\G)$, i.e.
\begin{equation}
\mathcal{S}'(\G):=[C^{\infty}_{\D}(\G)]^*.
\end{equation}
Let us denote by $\mathcal{S}'(\Gh)$ the space $[\mathcal{S}(\Gh)]^*$ of all linear linear continuous functionals on $\mathcal{S}(\Gh)$, i.e.
\begin{equation}
\mathcal{S}'(\Gh):=[\mathcal{S}(\Gh)]^*.
\end{equation}
 \end{defn}
 \begin{defn} For any distribution $u\in \mathcal{S}'(\Gh)$ its Fourier transform $\widehat{u}$ is a distribution on $C^{\infty}_{\D}(\G)$ given by
\begin{equation}
\widehat{u}(\widehat{f})
:=
u(f),\quad \widehat{f}\in\mathcal{S}(\Gh),\,f\in C^{\infty}_{\D}(\G).
\end{equation}
\end{defn}

 \begin{prop} A linear function $u$ on $C^{\infty}_{\D}$ is a distribution,  if and only if, there exists a constant $C$ and a number $\alpha>0$ such that
\begin{equation}
\left|
u(f)
\right|
\leq
C\left(\sum\limits_{\pi\in\Gh}
\dpi\npi
\left|\lambda_{\pi}\right|^{2\alpha}\|\widehat{f}(\pi)\|^2_{\HS}\right)^{\frac12},
\end{equation}
for every $f\in C^{\infty}_{\D}(\G)$.
 \end{prop}

\begin{prop} The space $\mathcal{S}'(\G)$ is complete, i.e. for every sequence $\{u_n\}\subset \mathcal{S}'(\G)$ the limit
\begin{equation}
u =\lim u_n \in \mathcal{S}'(\G)
\end{equation} exists and belongs to $\mathcal{S}'(\G)$.

If $\varphi_n$ converges to $\varphi$ in $C^{\infty}_{\D}$, then
\begin{equation}
\lim_n u_n(\varphi_n)=u(\varphi).
\end{equation}
\end{prop}

 By transposing the inverse Fourier transform $\FT^{-1}_{\G}\colon \mathcal{S}(\Gh)\to C^{\infty}_{\D}(\G)$, the Fourier transform $\FT_{\G}$ extends uniquely to the mapping
 \begin{equation}
\FT_{\G}\colon \mathcal{S}'(\G)\to \mathcal{S}'(\Gh)
 \end{equation}
 by the formula
 \begin{equation}
 \FT_{\G}[u](\sigma)=u(\FT^{-1}_{\G}[\sigma]),\quad u\in\mathcal{S}'(\G).
 \end{equation}
 In other words, for every distribution $u\in\mathcal{S}'(\G)$ its Fourier transform $\FT_{\G}[u]$ is a distribution on $\mathcal{S}(\Gh)$.
\begin{defn} For any distribution $u\in \mathcal{S}'(\G)$ its Fourier transform $\widehat{u}$ is a distribution on $\mathcal{S}'(\Gh)$ given by
\begin{equation}
\widehat{u}(\sigma)
:=
u(\FT^{-1}_{\G}(\sigma)),\quad \sigma\in \mathcal{S}(\Gh).
\end{equation}
\end{defn}

\begin{prop} 
\label{PROP:nuclearity}
Let $(\G,\Delta)$ be a compact quantum group and let $(C^{\infty}_{\D},L^2(\G),\D)$ be a summable spectral triple. Then the Frechet space $C^{\infty}_{\D}$ is nuclear.
\end{prop}
\begin{proof}[Proof of Proposition \ref{PROP:nuclearity}]
It is sufficient to prove that $\mathcal{S}(\Gh)$ is a nuclear Frechet space since $\FT_{\G}$ is a homeomorphism. The former fact follows from \cite[Section 8.2.1]{Triebel1978}.
\end{proof}
The theory of topological vector spaces has been significantly developed \cite{Grothendieck1955} by Alexander Grothendieck. 
It turns out that the property of being nuclear is crucial and these spaces are 'closest' to finite-dimensional spaces.  The nuclearity is the necessary and sufficient condition for the existence of abstract Schwartz kernels. 
The topological tensor product preserves nuclearity \cite{Treves1967}.
\begin{defn} A distribution $u$ on $C^{\infty}_{\D}\widehat{\otimes}C^{\infty}_{\D}$ is a linear continuous functional on $C^{\infty}_{\D}\widehat{\otimes}C^{\infty}_{\D}$. We denote by $\mathcal{S}'(\G)\widehat{\otimes}\mathcal{S}'(\G)$ the space of distributions $u$ on $C^{\infty}_{\D}\widehat{\otimes}C^{\infty}_{\D}$, i.e.
\begin{equation}
\mathcal{S}'(\G)\widehat{\otimes}\mathcal{S}'(\G)
:=
\left[
C^{\infty}_{\D}\widehat{\otimes}C^{\infty}_{\D}
\right]^{*}.
\end{equation}
\end{defn}
\begin{defn} A linear continuous operator $A\colon C^{\infty}_{\D}(\G)\to \mathcal{S}'(\G)$ is called a pseudo-differential operator.
\end{defn}
From the abstract Schwartz kernel theorem \cite{Treves1967}, we readily obtain
\begin{thm} Let $A\colon C^{\infty}_{\D}\to \mathcal{S}'(\G)$ be a pseudo-differential operator. Then there is a distribution $K_A\in\mathcal{S}'(\G)\widehat{\otimes}\mathcal{S}'(\G)$ such that
\begin{equation}
C^{\infty}_{\D}\ni \varphi \mapsto A\varphi \in \mathcal{S}'(\G),\quad \varphi\to (A\varphi)(\psi)=K_A(\varphi\otimes \psi),\quad \psi\in C^{\infty}_{\D}.
\end{equation}
\end{thm}
The structure theorem 
\cite[Theorem 45.1]{Treves1967}
applied to the topological tensor product 
$\mathcal{S}'(\G)\widehat{\otimes}\mathcal{S}'(\G)$
 immediately yields that the Schwartz kernel $K_A$ can be written in the form
\begin{equation}
\label{EQ:s-A}
K_A=\sum\limits^{\infty}_{n=1}s^A_n x^A_n\otimes t^A_n,
\end{equation}
where $\sum\limits^{\infty}_{n=1}\left|s^A_n\right|<+\infty$ and $\{x_n\}, \{t_n\}\subset\mathcal{S}'(\G)$ tend to $0$ in $\mathcal{S}'(\G)$.
This allows us to define global symbols $\sigma_A$ in line with the classical theory.
\begin{defn} Let $A\colon C^{\infty}_{\D}(\G)\to \mathcal{S}'(\G)$ be a pseudo-differential operator. We define a global symbol $\sigma_A$ of $A$ at $\pi\in\Gh$ as a distribution $ \sigma_A(\pi)\in\mathcal{S}'(\G)$ acting by the formula
\begin{equation}
\label{DEF:full_symbol-0}
C^{\infty}_{\D}(\G)\ni \varphi \mapsto \sigma_A(\pi)(\varphi)=K_A(\varphi\otimes\pi)\in\C.
\end{equation}
\end{defn}

Alternatively, we have
\begin{equation}
\label{DEF:full_symbol-1}
\sigma_A(\pi)_{ij}
=
\sum\limits^{\infty}_{n=1}s^A_n x^A_n h(t^A_n\pi^*_{ji})\in \mathcal{S}'(\G).
\end{equation}

\begin{defn}
We say that a pseudo-differential operator $A\colon C^{\infty}_{\D}(\G)\to \mathcal{S}'(\G)$ is regular if $K_A\in C^{\infty}_{\D}\widehat{\otimes}C^{\infty}_{\D}$.
\end{defn}

Explicit composition formula for the global symbols on quatum tori $\mathbb{T}^n_{\theta}$ has been recently obtained in \cite{Levy2016}. 
It can be easily seen that 

 this class of pseudo-differential operators is closed under composition.

Let $L^1(\G)$ be a predual of $L^{\infty}(\G)$. For two elements $a,b\in L^1(\G)$, we define their convolution
\begin{equation}
a \ast b \colon L^1(\G)\ni a \mapsto a \ast b= a\otimes b(\Delta )\in L^1(\G),
\end{equation}
where $a\otimes b(\Delta )$ is a functional on $\G$, i.e.
\begin{equation}
\G\ni a \ast b(\varphi)=a\otimes b(\sum\limits \varphi_{(1)}\otimes \varphi_{(2)})\in \C,
\end{equation}
where $\Delta\varphi=\varphi_{(1)}\otimes \varphi_{(2)}$.

We introduce the right-convolution Schwartz kernel $R_A$ by the formula
\begin{equation}
R_A
=
\sum\limits^{\infty}_{n=0}s^A_n x^A_n \otimes u^A_n,
\end{equation}
where $u^A_n$ is a convolution type vector-valued distribution acting by the formula
$$
C^{\infty}_{\D}\ni \varphi
\to 
u^A_n(\varphi)
=
(1\otimes h)((1\otimes u^A_n )(1\otimes S)\Delta \varphi)\in C^{\infty}_{\D},
$$
and $s^A_n$ are as in \eqref{EQ:s-A}.
\begin{thm} 
\label{THM:right-quantization}
Let $(\G,\Delta)$ be a compact quantum group and let $A\colon C^{\infty}_{\D}\to C^{\infty}_{\D}$ be a regular pseudo-differential operator acting via right-convolution kernel, i.e.
\begin{equation}
C^{\infty}_{\D}(\G)\ni f\mapsto 
Af
=
\sum\limits^{\infty}_{n=0}s^A_n x^A_n \otimes u^A_n(f)
\in C^{\infty}_{\D}(\G).
\end{equation}
Then we have
\begin{equation}
\label{EQ:right-quantization}
Af
=
\sum\limits_{\pi\in\Gh}\dpi \Tr
\left(
\sigma_A(\pi)\widehat{f}(\pi)\pi
\right)
,
\end{equation}
where $\sigma_A(\pi)$ is the global symbol of $A$ defined by \eqref{DEF:full_symbol-0}.
\end{thm}
\begin{proof}[Proof of Propositon \ref{THM:right-quantization}]

Let $f\in \C[\G]$. Then we have
\begin{equation*}
f
=
\sum\limits_{\pi\in I_f}\dpi\sum\limits^{n_{\pi}}_{i,j=1}\widehat{f}(\pi)_{ij}\pi_{ji}.
\end{equation*}
We shall start by showing that \eqref{EQ:right-quantization} holds true for $f\in \C[\G]$.
We have
\begin{align*}
Af
&=
\sum\limits^{\infty}_{n=0}s^A_nx^A_n \otimes u^A_n(f)
=
\sum\limits^{\infty}_{n=0}s^A_nx^A_n \otimes (1\otimes h)(u^A_n\Delta f)
\\
&=
\sum\limits^{\infty}_{n=0}s^A_nx^A_n \otimes (1\otimes h)((1\otimes u^A_n)\cdot  \sum\limits_{\pi\in I_f}\dpi\sum\limits^{n_{\pi}}_{i,j=1}\widehat{f}(\pi)_{ij}(1\otimes S)\Delta\pi_{ji})
\\
&=
\sum\limits_{\pi\in I_f}
\sum\limits^{n_{\pi}}_{i,j=1}
\sum\limits^{n_{\pi}}_{k=1}
\sum\limits^{\infty}_{n=0}
s^A_nx^A_n \otimes (1\otimes h)\left[
\dpi
\widehat{f}(\pi)_{ij}
\pi_{jk}\otimes u^A_n\pi^*_{ki}\right]
\\
&=
\sum\limits_{\pi\in I_f}
\dpi
\sum\limits^{n_{\pi}}_{i,j=1}
\sum\limits^{n_{\pi}}_{k=1}
\widehat{f}(\pi)_{ij}
\pi_{jk}
\cdot
\sum\limits^{\infty}_{n=0}
s^A_nx^A_n
h(u^{A}_n\pi^*_{ki})
\\
&=
\sum\limits_{\pi\in I_f}
\dpi
\sum\limits^{n_{\pi}}_{i,j=1}
\sum\limits^{n_{\pi}}_{k=1}
u^A_n
\widehat{f}(\pi)_{ij}
\pi_{jk}
\cdot
\sum\limits^{\infty}_{n=0}
s^A_nx^A_n
h(u^{A}_n\pi^*_{ki})
=
\sum\limits_{\pi\in I_f}
\dpi
\Tr
\left[
\widehat{f}(\pi)\pi\sigma_A(\pi)
\right].
\end{align*}
\end{proof}
\section{Differential calculi on compact quantum groups}
In this section we are going to ask how the above `Fourier approach' to the analysis on compact quantum groups interplays with the theory of differential structures on Hopf algebras of compact quantum groups and how this extends to $C^{\infty}_{\D}(\G)$. Recall that differential structures in the literature have been defined at the polynomial level i.e. on $\C[\G]$ as a Hopf $*$-algebra. For every choice of $\{\lambda_\pi\}_{\pi\in \hat\G}$ we have a bare spectral triple and
\[ \C[\G]\subseteq C^{\infty}_{\D}(\G) \subseteq \G.\]
Thus for $\G=\SUq2$ we have $\C[\G]=\C_q[SU_2]$ as the usual dense Hopf $*$-subalgebra of $\SUq2$ with a $2\times 2$ matrix of generators while $C^\infty_{\D}(\SUq2)$ lies in between as something more akin to $C_q^\infty(SU_2)$. Our goal in this section is to show that elements of $C^{\infty}_{\D}(\G)$ are indeed smooth with respect to a suitable differential structure at least for $\SUq2$ and in outline for the general $q$-deformation case.

We start recalling the purely algebraic definition of first-order differential calculus over associative algebras and refer to \cite{Majid2016} for a thorough exposition. Let $A$ be a unital algebra over a field $k$. 
\begin{defn} A first order differential calculus $(\Omega^1,\extd)$ over $A$ means
\begin{enumerate}
\item $\Omega^1$ is an $A$-bimodule.
\item A linear map $\extd\colon A\to \Omega^1$ satisfies
\begin{equation*}
\extd(ab)=(\extd a)b+a\extd b,\quad \forall a,b\in A.
\end{equation*}
\item The vector space $\Omega^1$ is spanned by $adb$
\begin{equation*}
\Omega^1=\Span\{a\extd b\}_{a,b\in A}.
\end{equation*}
\end{enumerate}
\end{defn}
In the $*$-algebra case $*$ extends uniquely to $\Omega^1$ in such a way that it commutes with $\extd$. 

\begin{ex} Let $A=C^{\infty}(\mathbb{R})$ and $\Omega^1=C^\infty(\RR).\extd x$ with left and right action given by 
multiplication in $C^\infty(\RR)$ (so functions and $\extd x$ commute). The exterior derivative is $\extd f={\del f\over\del x}\extd x$
as this is the classical calculus. 
\end{ex}
There are many other interesting calculi even on the commutative algebra of functions in one variable, see \cite{Majid2016}.

\begin{defn} A differential calculus $(\Omega^1,\extd)$ over a Hopf algebra $A$ is called left-covariant if:
\begin{enumerate}
\item There is a left coaction $\Delta_{L}\colon\Omega^1\to A\otimes\Omega^1$.
\item $\Omega^1$ with its given left action becomes a left Hopf module in the sense $\Delta_L(a\omega)=(\Delta a).(\Delta_L\omega)$
\item The exterior derivative $\extd\colon A\to \Omega^1$ is a comodule map, where $A$ coacts on itself by $\Delta$
\end{enumerate}
\end{defn}

This case was first analysed in \cite{Wor89} but here we continue with a modern algebraic exposition. Note that the last two requirements imply that $\Delta_L(a\extd b)=a\o b\o\tens a\t\extd b\t$ and conversely if this formula gives a well-defined map then one can show that it makes the calculus left covariant. Hence this is a property of $(\Omega^1,\extd)$ not additional data. 
We have a similar notion of right covariance and the calculus is called bicovariant if it is both left and right covariant. 
Let $\Lambda^1=\{\omega\in \Omega^1\ |\ \Delta_L\omega=1\tens\omega\}$ be the space of left-invariant 1-forms on a left-covariant calculus.

 In this case we define the Maurer-Cartan form $\varpi:A^+\to \Lambda^1$ by 
 $$\varpi(a)=Sa\o\extd a\t$$.
  This map is surjective by the spanning assumption above and is a right $A$-module map, since 
\[ \varpi(ab)=(Sb\o)(Sa\o)(\extd a\t)b\t+(Sb\o)\eps(a)\extd b\t=\varpi(a)\ra b,\] where $\Lambda^1$ is a right module by $\omega\ra b=(Sb\o)\omega b\t$. Hence $\Lambda^1\isom A^+/I$ for some right ideal $I\subset A^+$. 
Conversely, given a right $A$-module $\Lambda^1$ and a surjective right-module map $\varpi:A^+\to \Lambda^1$ we can define a left covariant calculus by exterior derivative and bimodule relations
\[ \extd a=a\o \varpi\pi_\eps a\t,\quad \omega a=a\o (\omega\ra a\t),\quad\forall a\in A.\]
Here $\pi_\eps(a)=a-\eps(a)$ projects $A\to A^+$ and $\Omega^1=A.\Lambda^1$ is free as a left module. If the calculus is bicovariant then $\Lambda^1$ also has a right coaction making $\Lambda^1$ an object in the braided category $\CM^A_A$ of crossed $A$-modules (also called Radford-Drinfeld-Yetter modules). Here  $A^+$ for any Hopf algebra is also a right crossed module by 
\[ a\ra b=ab,\quad \Ad_R(a)=a\t\tens (Sa\o)a\th\]
(the right adjoint coaction) and  $\varpi:A^+\to \Lambda^1$ becomes a surjective morphism of right  crossed-modules in the bicovariant case.

Now let the calculus be bicovariant and $\Lambda^1$ of finite dimension $n$ as a vector space ($\Omega^1$ finite-dimensional over $A$) and  $\{e_i\}_{i=1}^n$ a basis of $\Lambda^1$ with $\{f^i\}$ a dual basis. Then the associated `left-invariant vector fields' (which are not necessarily derivations) are given by
\[ \del^i:A\to A,\quad \del^i a= a\o f^i(\varpi\pi_\eps a\t),\quad\forall a\in A\]
and obey $\Delta\del^i=(\id\tens \del^i)\Delta$ as in Definition~4.1 but on $A$ and $\extd a=\sum_i (\del^i a)e_i$. The global symbols $\sigma^i:A\to k$ defined by $\sigma^i=\<f^i,\varpi\pi_\eps(\ )\>$ can be recovered from $\del^i$ as $\sigma^i(a)=\eps\del^i a$ and can typically be realised as evaluation against some element $x^i$ of a dually paired `enveloping algebra' Hopf algebra and in this context we will write $\sigma^i=\sigma_{x^i}=\<x_i,\ \>$.  Similarly for each $i,j$ let $C_i{}^j(a)=a\o\<f^j,e_i \ra a\t\>,$ for $a\in A$, be  left-invariant operators encoding the bimodule commutation relations.They have no classical analogue (they would be the identity). Their symbols  $\sigma_i{}^m:A\to k$ defined by $\sigma_i{}^j(a)=\<f^j,e_i \ra a\>=\eps C_i{}^j(a)$ can typically be given by evaluation against elements $y_i{}^j$ of a dually paired Hopf algebra and this in this context we will write $\sigma_i{}^j=\sigma_{y_i{}^j}=\<y_i{}^j,\ \>$. It is these global symbols which we extract from the algebraic structure of the calculus and need in what follows.

Now let $(\G,\Delta)$ be a compact quantum group with dense $*$-Hopf subalgebra $A=\C[\G]$.
 We suppose that we have a left covariant calculus on $\C[\G]$ and remember from it the key information $\Lambda^1$ and the operators $\del^i,C_i{}^j$ defining the exterior derivative and bimodule relations respectively.

\begin{prop} 
\label{PROP:partial_derivative}
Let $(\G,\Delta)$ be a compact quantum group and let $(\Omega^1,d)$ be a $n$-dimensional left-covariant differential calculus over the  dense Hopf $*$-algebra $\C[\G]$ of $\G$. Then $\del^i,C_i{}^j$ extend to left-coinvariant operators $C^{\infty}_{\D}(\G)\to C^{\infty}_{\D}(\G)$ and define a differential calculus on the algebra $C^{\infty}_{\D}(\G)$ if and only if  there exists $\gamma>0$ such that
\begin{equation}
\label{EQ:admissible-lambda}
\max_{i,j=1,\dots,n}\{\|\sigma_{\del^i}(\pi)\|^2_{\HS},\|\sigma_{C^j_i}(\pi)\|^2_{\HS}\}\leq \left|\lambda_{\pi}\right|^{\gamma}. 
\end{equation}

The extension is given by $\Omega^1(C^{\infty}_{\D}(\G))=C^{\infty}_{\D}(\G)\tens\Lambda^1=C^{\infty}_{\D}(\G)\tens_{\C[\G]}\Omega^1$ with $\extd a=\sum_i(\del^i a)e_i$ and $e_i.a=\sum_j C_i{}^j(a)e_j$. 
\end{prop}
\begin{proof}[Proof of Proposition \ref{PROP:partial_derivative}]  
From the linearity of the exterior derivative $\extd$ in the Fourier expansion, we have 
\begin{equation}
\label{EQ:da}
\extd a
=
\sum\limits_{\pi\in\widehat{\G}}
\dpi\Tr( (Q^{\pi})^{-1}(\extd \pi)\widehat{a}(\pi)),
\end{equation}
where from the results above including Theorem~4.3 in the algebraic form on $\C[\G]$,
\begin{equation*}
\extd(\pi_{ij})
=
\sum\limits^{n}_{k=1}\partial^k(\pi_{ij})e_k,\quad \del^k\pi_{ij}=\sum_m \pi_{im}\sigma_{\del^k}(\pi)_{mj},\quad \sigma_{\del^k}(\pi)_{mj}= \sigma^k(\pi_{mj})=\pi_{mj}(x^k),
\end{equation*}
and where the last step is the matrix of the representation of a dually paired Hopf algebra defined by $\pi$ when such $x^k$ exist. 

Therefore, it is sufficient to check that $\del^k\colon C^{\infty}_{\D}\to C^{\infty}_{\D}$ are continuous linear maps with respect to the topology defined by seminorms \eqref{EQ:p-norms}.
By \cite[Proposition 7.7, p.64]{Treves1967}, the linear maps $\del^k$ act continuously in $C^{\infty}_{\D}(\G)$ if and only if for every $\alpha>0$ there is $\beta>0$ such that
\begin{equation}
\label{EQ:iff}
\sum\limits_{\pi\in\Gh}\dpi
|\lambda_{\pi}|^{2\alpha}
\|\widehat{\del^k(a)}\|^2_{\HS}
\leq
\sum\limits_{\pi\in\Gh}\dpi
|\lambda_{\pi}|^{2\beta}\|\widehat{a}\|^2_{\HS}
\end{equation}
for every $a\in C^{\infty}_{\D}(\G)$.

It is clear that condition \eqref{EQ:admissible-lambda} implies \eqref{EQ:iff}. Hence, we concentrate on necessity. Taking $a=\pi_{ij}\in\Gh$ in \eqref{EQ:iff}, we get

\begin{equation}
\label{EQ:iff-1}
\lambda^{2\alpha}_{\pi}
\|
\widehat{\del^k(\pi_{ij})}(\pi)\|^2_{\HS}
\leq
\lambda^{2\beta}_{\pi}
\|\widehat{\pi_{ij}}(\pi)\|^2_{\HS}
=
\lambda^{2\beta}_{\pi}.
\end{equation}
From \eqref{EQ:iff-1} dividing by $|\lambda_{\pi}|^{2\alpha}$, we get
\begin{equation}
\label{EQ:iff-2}
\|
\widehat{\del^k(\pi_{ij})}(\pi)\|^2_{\HS}
\leq
|\lambda_{\pi}|^{2(\beta-\alpha)}.
\end{equation}
From the algebraic version of Theorem~4.3 we have
$$
\widehat{\del^k(\pi_{ij})}(\pi)_{mn}
=
\sum\limits^{n_{\pi}}_{s=1}\sigma_{\del^k}(\pi)_{ms}\widehat{\pi_{ij}}(\pi)_{sn}
=
\sigma_{\del^k}(\pi)_{mj}\frac{q^{\pi}_i}{\dpi}\delta_{ni},
$$
where we used 
\begin{equation*}
\widehat{\pi_{ij}}(\pi)_{sn}
=
\frac{q^{\pi}_i}{\dpi}\delta_{sj}\delta_{ni}.
\end{equation*}
The latter follows from the Peter-Weyl orthogonality relations \eqref{EQ:peter-weyl-orth}.
Hence, we get
\begin{equation}
\|\widehat{\del^k(\pi_{ij})}\|^2_{\HS}
=
\sum\limits^{n_{\pi}}_{m,n=1}\frac1{q^{\pi}_m}
\left|
\sigma_{\del^k}(\pi)_{mj}
\frac{q^{\pi}_i}{\dpi}
\delta_{ni}
\right|^2.
\end{equation}
Thus, estimate \eqref{EQ:iff-1} reduces to
$$
\|\sigma_{\del^k}(\pi)\|_{\HS}
\leq
\lambda^{\beta-\gamma}_{\pi},\quad \pi\in\Gh,
$$
with $\gamma=\beta-\alpha$.

We similarly need to extend the bimodule relations from $\C[\G]$ to $C^{\infty}_{\D}(\G)$ and we do this in just the same way by
\[ e_i.a = \sum\limits_{\pi\in\widehat{\G}}
\dpi\Tr((Q^\pi)^{-1}( e_i.\pi)\widehat{a}(\pi)),\]
where from the above and the algebraic form of Theorem~4.3  we have
\[ e_i.\pi_{kl}=\sum_jC_i{}^j(\pi_{kl})e_j,\quad C_i{}^j(\pi_{kl})=\sum_m\pi_{km}\sigma_{C_i{}^j}(\pi_{ml}),\quad \sigma_{C_i{}^j}(\pi_{ml})=\sigma_i{}^j(\pi_{ml})=\pi_{ml}(y_i{}^j),\]
and where the last step is the matrix of the representation of a dually paired Hopf algebra defined by $\pi$ when such $y_i{}^j$ exist. As before we need these linear maps $C_i{}^j:\C[\G]\to \C[\G]$ to extend to $\C^\infty_{\D}$ which is another Hilbert-Schmidt condition on the symbols of the same type as for the $\del^i$.  \end{proof}

Conversely, given a differential calculus $(\Omega^1,\extd)$ over the Hopf-subalgebra $\C[\G]$ of $\G$, we shall view \eqref{EQ:admissible-lambda} as a restriction on the $\{\lambda_\pi\}$ i.e. on a `Dirac operator' $\D\colon L^2(\G)\to L^2(\G)$ for it to agree with the differential calculus $(\Omega^1,\extd)$ over $\C[\G]$.

\begin{defn} Let $\D\colon L^2(\G)\to L^2(\G)$ be a `Dirac operator' in the sense of Theorem 
\ref{THM:D-a-bounded} defined by $\{\lambda_\pi\}_{\pi\in\Gh}$. We shall say that $\D$ is admissible with respect to a differential calculus $(\Omega^1,\extd)$ on $\mathbb{C}[\G]$ if and only if the condition \eqref{EQ:beta-dimension} on $\{\lambda_{\pi}\}_{\pi\in\Gh}$ holds for some $\gamma>0$.
\end{defn}

Whether or not an admissible $\D$ exists depends on the quantum group and the calculus. We look at $\SUq2$ with its two main calculi of interest, the 3D and the 4D (both of these calculi are from \cite{Wor89} but the 4D one generalises to other $q$-deformation quantum groups). 
As a first step, we recall the computation of the global symbols for the vector fields on the classical $\SU2$. 
We shall briefly recall representation theory of $\SUq2$ \cite{Masuda1991}. 
The unitary dual $\widehat{\SUq2}$ is parametrised by the half-integers $\frac12\NN_0$, i.e.
$$
\widehat{\SUq2}
=
\{t^l\}_{l\in\frac12\NN_0}.
$$
The Peter-Weyl theorem obtained in \cite[Theorem 3.7]{Masuda1991} allows us to describe the Fourier transform explicitly. For each $a\in\C[\SUq2]$, we define its matrix-valued Fourier coefficient at $t^l$ by
\begin{equation}
\widehat{a}(l)
=
h(aSt^l),\quad \text{i.e.} \quad \widehat{a}(l)_{mn}=h(aSt^l_{mn}),
\end{equation}
where $t^l=(t^l_{mn})_{m,n\in I_l}$ and $I_l=\{-l,-l+1,\cdots,+l-1,+l\},\,l\in\frac12\NN_0$.
It is convenient to introduce $q$-traces to define the inverse Fourier transform
\begin{equation}
\tau_l(\sigma(l))=\sum\limits_{i\in I_l} q^{2i}\sigma(l)_ii.
\end{equation}
Moreover, the $q$-trace $\tau_l$ naturally leads to the $q$-hermitian inner form $(\sigma_1(l),\sigma_2(l))=\tau_l \left(\sigma_1(l)\sigma_2(l)^*\right)$. The Fourier inversion formula takes \cite[Theorem 3.10]{Masuda1991} the form
\begin{equation}
\label{EQ:Fourier-inversion-SUq2}
h(ab^*)
=
\sum\limits_{l\in\frac12\NN_0}
[2l+1]_q
\sum\limits_{i,k\in I_l}q^{2i}
\widehat{a}(l)_{ik}\overline{\widehat{b}(l)_{ik}}
\end{equation}
Let us denote by $C^{\infty}(\SU2)$ the space of infinitely differentiable functions on $\SU2$.
Let 
$
X_+
=
\begin{psmallmatrix}
0 & 0
\\
1 & 0
\end{psmallmatrix}
,
X_-
=
\begin{psmallmatrix}
0 & 1
\\
0 & 0
\end{psmallmatrix}
,
H
=
\frac12
\begin{psmallmatrix}
-1 & 0
\\
0 & +1
\end{psmallmatrix}
$ be a basis in the Lie algebra $\mathfrak{su}_2$ of $SU(2)$ with $[X_+,X_-]=H$ and associated first-order partial differential operators $\partial_{+},\partial_{-},\partial_+\colon C^{\infty}(SU(2))\to C^{\infty}(SU(2))$ (called {\it creation}, {\it annihilation} and {\it neutral} operators, respectively, in \cite{Ruzhansky+Turunen-IMRN}). Then classically, in our current conventions, one has the following.

\begin{prop}[{{\cite[Theorem 5.7, p.2461]{Ruzhansky+Turunen-IMRN}}}] 
\label{PROP:vector-fields-su2}

\begin{eqnarray*}
\begin{aligned}
\partial_+ t^l_{m n}
&=
\sqrt{(l-n)(l+n+1)}t^l_{m\, n+1},
\\
\partial_-t^l_{m n}
&=
\sqrt{(l+n)(l-n+1)}t^l_{m\,n-1},
\\
\partial_0 t^l_{mn}
&=
nt^l_{mn}.
\end{aligned}
\end{eqnarray*}
\end{prop}
From this the classical global symbols $\sigma_{\del_\pm},\sigma_{\del_0}$ can be read off as the matrix entries of $X_\pm,H$ in the representation $t^l$. The corepresentation theory of $\SUq2$ is strikingly similar to its classical counterpart giving similar results. We compute the symbols for the action of $X_{-},X_{+},q^{\frac{H}2}$ as elements of the quantum enveloping algebra $U_q(su_2)$ acting by the regular representation on $\C_q[SU_2]$ and in the conventions of \cite{Majid1995}. 
\begin{lem} 
\label{LEM:symbols_suq2}
We have 
\begin{align*}
\sigma_{X_+}(t^l)_{mn}
&=
\sqrt{[l-n]_q[l+n+1]_q}
\delta_{m\,n+1}
,
\\
\sigma_{X_-}(t^l)_{mn}
&=
\sqrt{[l+n]_q[l-n+1]_q}
\delta_{m\,n-1}
,
\\
\sigma_{q^{\frac{H}2}}(t^l)_{mn}
&=
q^n
\delta_{m\,n}
\end{align*}where $[n]_q=\frac{q^n-q^{-n}}{q-q^{-1}}$. 
\end{lem}
\begin{proof}[Proof of Lemma \ref{LEM:symbols_suq2}]

Let $q$ be real and for each $l\in\frac12\NN_0$, the quantum group $\mathrm{U}_q(su2)$ has $2l+1$-dimensional unitary representation space $V_l=\{\ket{l\,m}\}^{+l}_{m=-l}$ detailed in the relevant conventions in \cite[Proposition 3.2.6, p.92]{Majid1995} so that, for example, $X_+\ket{l,m}=\sqrt{[l-n]_q[l+n+1]_q}\ket{l,m+1}$. By definition, the $t^{l}_{mn}$ are the matrix elements of this representation, immediately giving $\sigma_X(t^l)_{mn}=\<X,t^l_{mn}\>=t^l(X)_{mn}$ for the symbol of any left-invariant operator $\tilde{X}(a)=a\o \<X, a\t\>$. Thus we can read off the $\sigma_{X_\pm},\sigma_{q^{\frac{H}2}}$ as stated.   \end{proof}

For the convenience of the reader we recall that the $Q$-matrix for $\G=\SUq2$ is given \cite{Masuda1991} by
\begin{equation}
Q^{l}=\diag(q^{-2i})^{l}_{i=-l},\quad l\in \frac12\NN_0.
\end{equation}
As a warm-up we look at the admissibility condition \eqref{EQ:admissible-lambda} of Proposition \ref{PROP:partial_derivative}.
\begin{lem}
\label{LEM:b_q}
Let $b_q=\max(q,\frac1q)$. Then 
\begin{equation*}
[n]_q \cong b^n_q.
\end{equation*}
\end{lem} 
We write $x\cong y$ if there are constants $c_1,c_2\neq 0$ such that
$$
c_1 x\leq y\leq c_2 x.
$$
\begin{lem}
\label{EX:admissibility-SUq2}
Let $\G=\SUq2$. 
Then we have
\begin{eqnarray*}
\|\sigma_{X_+}(t^l)\|_{\HS}
\lesssim
[2l+1]_q,
\\
\|\sigma_{X_-}(t^l)\|_{\HS}
\lesssim
[2l+1]_q,
\\
\|\sigma_{q^{\frac{H}2}}(t^l)\|_{\HS}
\lesssim
[2l+1]_q.
\end{eqnarray*}
\end{lem}
\begin{proof}[Proof of Lemma \ref{EX:admissibility-SUq2}] 
By \eqref{EQ:HS-norm}
\begin{align}
\begin{split}
\label{EQ:HS-norm-X+}
\|\sigma_{X_+}(t^l)\|^2_{\HS}
&:=
\sum\limits^{+l}_{m=-l}q^{2m}
\sum\limits^{+l}_{n=-l}
\left|
\sigma_{X_+}(t^l)_{mn}
\right|^2
=
\sum\limits^{+l}_{m=-l}
q^{2m}
[l-m+1]_q[l+m]_q 
\\ &\cong
\sum\limits^{+l}_{m=-l}
q^{2m}
b^{l-m+1}_q b^{l+m}_q
=
b^{2l+1}_q
\sum\limits^{+l}_{m=-l}
q^{2m}
=
b^{4l}_q,
\end{split}
\end{align}
where we used the fact that
\begin{equation}
\label{EQ:-l-sum-b}
\sum\limits^{+l}_{m=-l}q^m
\cong
b^{l}_q.
\end{equation}
Similarly, we get
\begin{align}
\begin{split}
\|\sigma_{X_-}(t^l)\|^2_{\HS}
&=
\sum\limits^{+m}_{m=-l}
q^{2m}[l+m+1]_q[l-m]_q
\cong
\sum\limits^{+m}_{m=-l}
q^{2m}b^{l+m+1}_q b^{l-m}_q
=
b^{2l+1}_q
b^{2l}_q
\\&\cong
b^{4l}_q\cong[2l+1]^2_q.
\end{split}
\end{align}

Finally, we compute 
\begin{equation}
\label{EQ:q-H-HS}
\|\sigma_{q^{\frac{H}2}}(t^l)\|^2_{\HS}
=
\sum\limits^{+l}_{m=-l}q^{2m}q^{2m}
\cong
b^{4l}_q
\cong
[2l+1]^2_q.
\end{equation}
This completes the proof.
\end{proof}
By the arguments as in the proof of Proposition~7.4 it follows that the associated left covariant operators to $X_\pm$ and $q^{H\over 2}-1$ extend to $C^\infty_D(\SUq2)$ where $\D$ has $\lambda_{t^l}=[2l+1]_q$ and $\gamma=2$. 
\subsection{3D calculus on $\SUq2$}

We are now ready for the left-covariant 3D calculus on $\C_q[SU_2]$ which we take with the defining 2-dimensional representation with $t^\alpha{}_\beta=\{a,b,c,d\}$ to give  the standard matrix of generators with usual conventions where $ba=qab$ etc. We let $|\ |$ denote the known $\Bbb Z$-grading on the algebra defined as the number of $a,c$ minus the number of $b,d$ in any monomial. The 3D calculus has generators $e_0,e_\pm$ with  commutation relations
\[ e_0 f = q^{2 |f|} f e_0,\quad e_\pm f = q^{|f|} f e_\pm\]
(which implies the action in the vector space $\Lambda^1$ with basis $e_0,e_\pm$). The exterior derivative is 
\[  \extd a=a e_0+q b e_+,\quad \extd
b=a e_--q^{-2}b e_0,\quad \extd c=c  e_0+q d e_+,\quad \extd d=c
e_--q^{-2}d e_0\] 

Next the combinations $\sigma^k=f^k\circ\varpi\pi_\eps=\eps\del^k$ are linear functionals on $A=\C_q[SU_2]$ and can in fact be identified as evaluation against  elements $x^k\in U_q(su_2)$ in our case. 

\begin{prop} 
\label{PROP:spinor}
The $3D$ calculus $(\Omega^1_{3D},\C_q[SU_2],\extd)$ over $\C_q[SU_2]$ is generated by the action of 
\begin{eqnarray*}
x^+=q^{1\over 2}X_-q^{H\over 2},\quad 
x^-=q^{-{1\over 2}}X_+q^{H\over 2},\quad 
x^0={q^{2H}-1\over q^2-1}
\end{eqnarray*}
and extends to $C^{\infty}_{\D}(\SUq2)$ where $\D$ is defined as classically by Example~5.8.
The symbols are given by
\begin{align*}
\sigma_{x^+}(t^l)_{mn}
&=
q^{n+\frac12}
\sqrt{[l+m+1]_q[l-m]_q}\delta_{m\,n+1}(1-\delta_{m\,2l+1}),
\\
\sigma_{x^-}(t^l)_{mn}
&=
q^{n-\frac12}
\sqrt{[l-m+1]_q[l+m]_q}
\delta_{m\,n}(1-\delta_{m\,1})_{mn},
\\
\sigma_{x^0}(t^l)_{mn}
&=
\frac{h(t^l_{mn})-q^{4n}\delta_{m\,n}}{1-q^2}.
\end{align*}
\end{prop}
\begin{proof}[Proof of Proposition \ref{PROP:spinor}]  The 3D calculus is constructed `by hand' so we use the form $\sigma^k=\eps\del^k$ and the known form of the partial derivatives (obtained by computing $\extd$ on monomials via the Leibniz rule) and find elements $x^k\in U_q(su_2)$ as stated that give these. One then finds the symbols $\sigma_{x^k}(\pi)_{ij}=\pi(x^i)_{ij}$ as 
\begin{align*}
&\sigma_{x^+}(t^l)_{mn}
=
q^{\frac12}[\sigma_{X_{-}}(t^l)\sigma_{q^{\frac{H}2}}(t^l)]_{mn}
=
q^{\frac12}
\sum\limits^{2l+1}_{k=1}
\sigma_{X_{-}}(t^l)_{mk}\sigma_{q^{\frac{H}2}}(t^l)_{kn}
\\
&\quad\quad=\sum\limits^{2l+1}_{k=1}
\sqrt{[l+k]_q[l-k+1]_q}\delta_{m\,k-1}
q^n\delta_{k\,n}
=
q^{n+\frac12}
\sqrt{[l+m+1]_q[l-m]_q}\delta_{m\,n+1}(1-\delta_{m\,2l+1}),
\end{align*}
where we used the fact that $\sigma_X(\pi)\sigma_Y(\pi)=\sigma_{XY}(\pi)$ since these are matrices for $X,Y,XY$ in the representation $\pi$,  and  Lemma \ref{LEM:symbols_suq2}. Similarly, we establish
\begin{align*}
\sigma_{x^-}(t^l)_{mn}
&=
q^{n-\frac12}
\sqrt{[l-m+1]_q[l+m]_q}
\delta_{m\,n}(1-\delta_{m\,1})_{mn},
\\
\sigma_{x^0}(t^l)_{mn}
&=
\frac{h(t^l_{mn})-q^{4n}\delta_{m\,n}}{1-q^2}.
\end{align*}
where for the haar function $h$ we estimate
\begin{equation}
\left|
h(t^l_{mn})
\right|
\leq
\|t^l_{mn}\|_{\op}
h(1)
\leq
h(1).
\end{equation}
It is then straightforward to check that the condition \eqref{EQ:admissible-lambda} is satisfied for the symbols $\sigma_{x^{\pm}}(t^l), \sigma_{x^0}(t^l)$. Hence, the application of Proposition \ref{PROP:partial_derivative} shows that the vector fields $x_{\pm}, x_0$ are continuous.

Now, we check condition \eqref{EQ:admissible-lambda} allowing us to extend $x^+,x^-,x^0$ continuously.
We have
\begin{align}
\begin{split}
\|\sigma_{x^+}(t^l)\|^2_{\HS}
&=
\sum\limits^{+l}_{m,n=-l}q^{2n+1}[l+m+1]_q[l-m]_q
\cong
\sum\limits^{+l}_{m,n=-l}q^{2n+1}
b^{l+m+1}_q
b^{l-m}_q
\\ &=
b^{2l+1}_q
\sum\limits^{+l}_{m,n=-l}q^{2n+1}
\cong
b^{2l+1}_qb^{2l}
\cong b^{4l}_q.
\end{split}
\end{align}
Similarly
$$
\|\sigma_{x^-}(t^l)\|^2_{\HS}
=
\sum\limits^{+l}_{m,n=-l}q^{2n-1}[l-m+1]_q[l+m]_q
\cong
b^{4l}_q.
$$
We similarly have commutation relations given for $i,j=\pm$ by
\[ \sigma_i{}^j(t^l)_{mn}=\delta_{ij}t^l(y^i)_{mn}=\delta_{ij}\delta_{mn}\begin{cases}q^{-2m} & i=\pm\\ q^{-4m} & i=0\end{cases};\quad y^\pm=q^{-H},\ y^0=q^{-2H}\]
if we number the indices by $-l,\cdots,l$ for the $2l+1$ dimensional representation $t^l$. This gives commutation relations
$e_it^l_{jk}=t^l_{jm}\sigma_i{}^p(t^l)_{mk}e_p= t^l_{jm}e_i\delta_{mk}q^{-2k}=t^l_{jk}q^{-2k}e_i$ for $i=\pm$ (and $q^2$ in place of $q$ if $i=0$) which 
corresponds to a $\ZZ$-grading of $\C_q[SU_2]$ where $t^l_{jk}$ has grade $-2k$. For the spin $1/2$ representation it means $a,c$ in the standard matrix generators $t_{mn}$ of the quantum group have grade 1 and $b,d$ have grade -1 as expected. We compute $\|\sigma_i{}^j(t^l)\|^2_{\HS}$ similarly as in the proof of Lemma \ref{EX:admissibility-SUq2}.
By \eqref{EQ:F-q-HS}
\begin{align}
\begin{split}
\|\sigma_i{}^j(t^l)\|^2_{\HS}
:=
\sum\limits^{+l}_{n=-l}q^{2n}
\sum\limits^{+l}_{m=-l}
\left|
q^{-2n}
\delta_{mn}\delta_{ij}
\right|^2
=
\delta_{ij}
\sum\limits^{+l}_{n=-l}
=
(2l+1)
\leq
b^{2l+1}_q
\cong [2l+1]_q.
\end{split}
\end{align}

\end{proof}

\subsection{4D calculus on $\SUq2$}

As before we denote the standard $2\times 2$ matrix of generators of $\C_q[SU_2]$ by $a,b,c,d$. This time (from the general construction given later or from \cite{Wor89})  there is a basis $e_a,e_b,e_c,e_d$ corresponding to the generators, with relations and exterior derivative 
\[ e_a
\begin{pmatrix}a&b\\ c&d\end{pmatrix}=\begin{pmatrix}qa&q^{-1} b\\
qc&q^{-1}d\end{pmatrix}e_a\]
\[  [e_b, \begin{pmatrix}a&b\cr c&d\end{pmatrix}]=q\lambda\begin{pmatrix}0&a\cr 0&c\end{pmatrix}e_a,
\quad [e_c, \begin{pmatrix}a&b\cr c&d\end{pmatrix}]=q\lambda\begin{pmatrix}b&0\cr d&0\end{pmatrix}e_a\]
\[ [e_d,\begin{pmatrix}a\cr c\end{pmatrix}]_{q^{-1}}=\lambda \begin{pmatrix}b\cr d\end{pmatrix}e_b,\quad 
[e_d,\begin{pmatrix}b\cr d\end{pmatrix}]_q=\lambda  \begin{pmatrix}a\cr c\end{pmatrix}e_c+q\lambda ^2 \begin{pmatrix}b\cr d\end{pmatrix}e_a,\]
\[ \extd \begin{pmatrix}a\\ c\end{pmatrix}=\begin{pmatrix}a\\ c\end{pmatrix}((q-1)e_a+(q^{-1}-1)e_d)+\lambda \begin{pmatrix}b\\ d\end{pmatrix} e_b\]
\[\extd \begin{pmatrix}b\\ d\end{pmatrix}=\begin{pmatrix}b\\ d\end{pmatrix}((q^{-1}-1+q\lambda^2)e_a+(q-1)e_d)+\lambda \begin{pmatrix}a\\ c\end{pmatrix} e_c.  \]
Here $[x,y]_q\equiv
xy-qyx$ and $\lambda=1-q^{-2}$. 

Then the generators of the calculus are elements $x^k\in U_q(su_2)$ where $k=a,b,c,d$ which we organise as a $2\times 2$ matrix of elements $(x^{\alpha\beta})$ where
\begin{equation} 
\label{EQ:x-alpha-beta}
(x^{\alpha\beta})=\begin{pmatrix}q^{H}+q\lambda^2X_-X_+ -1 & q^{\frac{1}{ 2}}\lambda X_- q^{-{H\over 2}}\\ q^{\frac{1}{ 2}}\lambda q^{-{H\over 2}} X_+& q^{-H}-1\end{pmatrix}
\end{equation}
These combinations $x^{\alpha\beta}$ are known to span the right handed braided-Lie algebra $L\subset U_q(su_2)$ and generate the quantum group\cite{Majid2015}. 

\begin{prop} 
\label{PROP:q-Laplace}
The $4D$ calculus $(\Omega^1_{4D},\C_q[SU_2],\extd)$ continuously extends to $C^{\infty}_{\D}(\SUq2)$.
\end{prop}
\begin{proof}[Proof of Proposition \ref{PROP:q-Laplace}]
We compute the symbol $\sigma_{x^{\alpha\beta}}(\pi)$ of $x^{\alpha\beta}$ composing the results in Lemma \ref{LEM:symbols_suq2} to find
\begin{equation} 
\label{EQ:sigma-x-alpha-beta}
\sigma_{(x^{\alpha\beta})}(t^l_{mn})
=
\begin{pmatrix}
(q^{2l}+q^{-2l-2}-q^{-2n-2}-1)\delta_{mn}  
&  
q^{-n+\frac{1}{ 2}}\lambda \sqrt{[l+n]_q[l-n+1]_q}\delta_{m\,n-1}
\\ 
q^{-n-\frac12}\lambda\sqrt{[l-n]_q[l+n+1]_q}\delta_{m\,n+1}
& 
(q^{-2n}-1)\delta_{mn}
\end{pmatrix}
\end{equation}

It is sufficient to check that $x^{\alpha\beta}$ acts continuously in each $L^2(\G)$.  Let us denote
$$
\sigma_{(x^{\alpha\beta})}(t^l)
=
\begin{pmatrix}
\sigma^a(t^l)  &  \sigma^b(t^l)\\ 
\sigma^c(t^l)
& \sigma^d(t^l)
\end{pmatrix}.
$$

By \eqref{EQ:HS-norm}
\begin{equation}
\label{EQ:F-q-HS}
\|\sigma(t^l)\|^2_{\HS}
=
\sum\limits^{+l}_{m=-l}q^{2m}
\sum\limits^{+l}_{n=-l}
\left|
\sigma(t^l)_{mn}
\right|^2.
\end{equation}
Composing \eqref{EQ:F-q-HS} and \eqref{EQ:sigma-x-alpha-beta}, we get
\begin{align*}
\begin{split}
\|\sigma^a(t^l)\|^2_{\HS}
&=
\sum\limits^{+l}_{m=-l}
q^{2m}
\sum\limits^{+l}_{n=-l}
\left|
\sigma^a(t^l)_{mn}
\right|^2
=
\sum\limits^{+l}_{m=-l}
q^{2m}
\left(
q^{2l}
+
q^{-2l-2}
-q^{-2m-2}
-1
\right)^2
\\
&=
\sum\limits^{+l}_{m=-l}
(q^2l+q^{-2l-2})^2
-
2(q^{2l}+q^{-2l-2})(q^{2m-2}-1)+(q^{-2m-2}-1)^2
\\
&=\sum\limits^{+l}_{m=-l}
\big(
q^{4l}+2q^{2l-2l-2}+q^{-4l-4}
-
2(q^{2l-2m-2}-q^{2l}+q^{-2l-2m-4}-q^{-2l-2})
\\
&\quad\quad\quad\quad +
q^{-4m-4}-2q^{-2m-2}+1\big)
\\
&=
(q^{4l}+2q^{-2}+q^{_4l-4}+2q^{2l}+2q^{-2l-2}+1)(2l+1)
\\
&\quad\quad+
(-2q^{2l-2}-2q^{-2l-4}-2q^{-2})
\sum\limits^{+l}_{m=-l}q^{-2m}
+
q^{-4}
\sum\limits^{+l}_{m=-l}q^{-4m}.
\end{split}
\end{align*}
The expression $-2q^{2l-2}-2q^{-2l-4}-2q^{-2}$ is always negative. Therefore, we get
\begin{equation}
\label{EQ:sigma-a-1}
\|\sigma^a(t^l)\|^2_{\HS}
\leq
\left(q^{4l}+2q^{-2}+q^{-4l-4}+2q^{2l}+2q^{-2l-2}+1\right)(2l+1)
+
q^{-4}
\sum\limits^{+l}_{m=-l}q^{-4m}.
\end{equation}
It is straightforward to check that
\begin{equation}
\label{EQ:sum-of-q}
q^{4l}+2q^{-2}+q^{-4l-4}+2q^{2l}+2q^{-2l-2}+1
\cong b^{4l}_q
\end{equation}
and
\begin{equation}
\label{EQ:sum-l-plus}
\sum\limits^{+l}_{m=-l}q^{-4m}
=
\sum\limits^{+l}_{m=-l}q^{4m}
\cong b_q^{4l}.
\end{equation}
Using \eqref{EQ:sum-of-q} and \eqref{EQ:sum-l-plus}, we get from \eqref{EQ:sigma-a-1}
\begin{equation}
\|\sigma^a(t^l)\|^2_{\HS}
\lesssim
b^{4l}_q(2l+1)
\lesssim
b^{5l}_q
\cong
\left(
b^{2l+1}_q
\right)^{\frac{5}{2}}
\cong
[2l+1]^{\frac{5}{2}}_q,
\end{equation}
where in the first inequality we used the fact
$$
(2l+1)\lesssim b^{l}_q.
$$
In the second inequality we used 
Lemma \ref{LEM:b_q} with $n=2l+1$.
Now, we compute $\|\sigma^b(t^l)\|^2_{\HS}$
\begin{align}
\begin{split}
\|\sigma^b(t^l)\|^2_{\HS}
&=
\sum\limits^{+l}_{m=-l}
q^{2m}
\sum\limits^{+l}_{n=-l}
\left|
q^{-n+\frac12}\lambda
\sqrt{
[l+n]_q
[l-n+1]_q
}
\delta_{m,n-1}
\right|^2
\\&=
\sum\limits^{+l}_{m=-l}
q^{2m}q^{-2m}q^{-1}\lambda^2
[l+m+1]_q[l-m]_q
\cong
\sum\limits^{+l}_{m=-l}
[l+m+1]_q[l-m]_q
\\
&\cong
\sum\limits^{+l}_{m=-l}
b_q^{l+m+1}b^{l-m}_q
=
\sum\limits^{+l}_{m=-l}
b^{2l}_q
=
(2l+1)
b^{2l}_q
\lesssim b^{4l}_q
\cong
[2l+1]^2_q.
\end{split}
\end{align}
We can argue analogously for $\|\sigma^c(t^l)\|_{\HS}$ to get
\begin{equation}
\|\sigma^b(t^l)\|^2_{\HS}
\lesssim[2l+1]^2_q.
\end{equation}
Finally, one checks by direct calculation that
\begin{align}
\begin{split}
\|\sigma^d(t^l)\|^2_{\HS}
&=
\sum\limits^{+l}_{m=-l}q^{2m}\left|(q^{-2n}-1)\delta_{mn}\right|^2
=
\sum\limits^{+l}_{m=-l}
q^{2m}
\left(
q^{-4m}
-
2q^{-2m}
+1
\right)
\\&=
\sum\limits^{+l}_{m=-l}q^{-2m}
-2(2l+1)
+
\sum\limits^{+l}_{m=-l}q^{2m}
\leq
2
\sum\limits^{+l}_{m=-l}q^{2m}
\cong
b^{2l}_q
\cong [2l+1]_q.
\end{split}
\end{align}
Now, we check that the matrices $\sigma_{\alpha\beta}{}^{\gamma\delta}$ encoding the bimodule commutation relations in the $4D$ calculus satisfy condition \eqref{EQ:admissible-lambda} with some exponent $\gamma$. The bimodule relations are best handled as part of a general construction discussed later and from  \eqref{EQ:sigma-alpha-beta-gamma-delta} and \eqref{EQ:l-pm} there, we see the seven values
\begin{align*}
\sigma_{j1}{}^{i2}(t^l)_{mn}
=
\sigma_{(Sl^{-i}{}_j) l^{+1}{}_2}(t^l)_{mn}
=0,\quad  \sigma_{1i}{}^{2j}(t^l)_{mn}
=
\sigma_{(Sl^{-2}{}_1) l^{+i}{}_j}(t^l)_{mn}=0
\end{align*}
since $l^{+1}{}_2=l^{-2}{}_1=0$. The non-zero matrices $\sigma_{\alpha\beta}{}^{\gamma\delta}$ are obtained by reading \eqref{EQ:l-pm} and plugging it into \eqref{EQ:sigma-alpha-beta-gamma-delta}, noting that $SX_-=-q^{-1}X_-$ $Sq^{H\over 2}=q^{-{H\over 2}}$ for the action of the antipode. We then compute the symbols by composing the symbols for the composition of invariant operators, to obtain 
\begin{align*}
\label{EQ:sigma-12}
\sigma_{11}{}^{11}(t^l)_{mn}&=\sigma_{
(Sl^{-1}{}_1)
l^{+1}{}_1
}(t^l)_{mn}
=
\sigma_{q^{\frac{H}2}q^{\frac{H}2}}(t^l)_{mn} =q^{2n}\delta_{mn},
\\
\sigma_{12}{}^{11}(t^l)_{mn}
&=\sigma_{(S
l^{-1}{}_1)
l^{+2}{}_1}
(t^l)_{mn}
=
\sigma_{q^{\frac{H}2}q^{-\frac12}(q-q^{-1})X_+}(t^l)_{mn}
\\ &=
q^{-\frac12}(q-q^{-1})
q^{m}
\sqrt{
[l-m]_q[l+m+1]_q
}
\delta_{m,n+1},
\\
\sigma_{12}{}^{12}(t^l)_{mn}
&=\sigma_{(S
l^{-1}{}_1)
l^{+2}{}_2
}(t^l)_{mn}
=
\sigma_{q^{\frac{H}2}\cdot q^{-\frac{H}2}}(t^l)_{mn}=
\delta_{mn},
\\
\sigma_{21}{}^{11}(t^l)_{mn}&=\sigma_{
(Sl^{-1}{}_2)
l^{+1}{}_1
}(t^l)_{mn}
=
\sigma_{q^{-\frac12}(q-q^{-1})X_{-}\cdot q^{\frac{H}2}}(t^l)_{mn}
\\
&=
q^{-\frac12}(q-q^{-1})
\sqrt{[l+n]_q[l-n+1]_q}
q^n\delta_{m,n-1},
\\
\sigma_{22}{}^{11}(t^l)_{mn}
&=\sigma_{
(Sl^{-1}{}_2)
l^{+2}{}_1
}(t^l)_{mn}
=
\sigma_{q^{-\frac12}(q-q^{-1})X_{-}\cdot q^{-\frac12}(q-q^{-1})X_+}(t^l)_{mn}
\\
&=
q^{-1}(q-q^{-1})^2
\sqrt{[l+m]_q[l-m+1]_q}
\sqrt{[l-n]_q[l+n+1]_q}
\delta_{mn},
\\
\sigma_{22}{}^{12}(t^l)_{mn}
&=\sigma_{
(Sl^{-1}{}_2)
l^{+2}{}_2
}(t^l)_{mn}
=\sigma_{q^{-\frac12}(q-q^{-1})X_{-}\cdot q^{-{H\over 2}}}(t^l)_{mn}
\\
&=q^{-\frac12}(q-q^{-1})
\sqrt{[l+n]_q[l-n+1]_q}
q^{-n}\delta_{m,n-1}
\\
\sigma_{21}{}^{21}(t^l)_{mn}
&=\sigma_{
(Sl^{-2}{}_2)
l^{+1}{}_1
}(t^l)_{mn}
=
\sigma_{q^{-\frac{H}2}\cdot q^{\frac{H}2}}(t^l)_{mn}
=\delta_{mn},
\\
\sigma_{22}{}^{21}(t^l)_{mn}
&=\sigma_{
(Sl^{-2}{}_2)
l^{+2}{}_1
}(t^l)_{mn}
=
\sigma_{q^{-\frac{H}2}\cdot q^{-\frac12}(q-q^{-1})X_+}(t^l)_{mn}
\\
&=
q^{-\frac12}(q-q^{-1})q^{-m}\sqrt{[l-n]_q[l+n+1]_q}\delta_{m,n+1},
\\
\sigma_{22}{}^{22}(t^l)_{mn}
&=\sigma_{
(Sl^{-2}{}_2)
l^{+2}{}_2}(t^l)_{mn}
=
\sigma_{
q^{-\frac{H}2}q^{-\frac{H}2}
}(t^l)_{mn}=
q^{-2n}
\delta_{mn}.
\end{align*}
Now we can compute the corresponding $q$-deformed Hilbert-Schmidt norms,  
\begin{align*}
\|\sigma_{11}{}^{11}(t^l)\|^2_{\HS}
&=
\sum\limits^{+l}_{-l}q^{6m}
\cong
b_q^{6l}
\cong
[2l+1]^3_q,
\\
\|\sigma_{12}{}^{11}(t^l)\|^2_{\HS}
&\cong
\sum\limits^{+l}_{m=-l}
q^{2m}q^{2m}
[l-m]_q
[l+m+1]_q
\cong
\sum\limits^{+l}_{m=-l}q^{4m} b^{l-m}_q b^{l+m+1}_q
\cong
b^{4l}_q b^{2l+1}_q
\cong
[2l+1]^3_q,
\\
\|\sigma_{12}{}^{12}(t^l)\|^2_{\HS}&
=
\|\sigma_{21}{}^{21}(t^l)\|^2_{\HS}
\cong
\sum\limits^{+l}_{m=-l}q^{2m}
\cong b^{2l}_q
\cong [2l+1]_q,
\\
\|\sigma_{21}{}^{11}(t^l)\|^2_{\HS}&
\cong
\sum\limits^{+l}_{m=-l}q^{2m}q^{2m}[l+m+1]_q[l-m]_q
\cong
b^{2l+1}_q
\sum\limits^{+l}_{m=-l}q^{4m}
\cong
b^{2l+1}_q
b^{4l}_q
\cong
[2l+1]^3_q,
\\
\|\sigma_{22}{}^{11}(t^l)\|^2_{\HS}&
\cong
\sum\limits^{+l}_{m=-l}
q^{2m}
[l+m]_q[l-m+1]_q[l-m]_q[l+m+1]_q
\\ &\cong
\sum\limits^{+l}_{m=-l}
q^{2m}
b^{l+m}_q b^{l-m+1}_q b^{l-m}_q b^{l+m+1}_q
=
b^{2(2l+1)}_q
b^{2l}_q
\lesssim b^{3(2l+1)}_q
\cong
[2l+1]^3_q,
\end{align*}
\begin{align*}
\|\sigma_{22}{}^{12}(t^l)\|^2_{\HS}&
\cong
\sum\limits^{+l}_{m=-l}
q^{2m}
q^{-2m}
[l+m+1]_q[l-m]_q
\cong
\sum\limits^{+l}_{m=-l}
b^{l+m+1}_qb^{l-m}_q
=
(2l+1)b^{2l+1}_q
\lesssim
[2l+1]^2_q,
\\
\|\sigma_{22}{}^{21}(t^l)\|^2_{\HS}
&
\cong
\sum\limits^{+l}_{m=-l}
q^{2m}q^{-2m}
[l-m+1]_q[l+m]_q
\cong
\sum\limits^{+l}_{m=-l}
b^{l-m+1}_qb^{l+m}_q
=
(2l+1)b^{2l+1}
\lesssim
[2l+1]^2_q,
\\
\|\sigma_{22}{}^{22}(t^l)\|^2_{\HS}
&=
\sum\limits^{+l}_{m=-l}
q^{2m}q^{-4m}
=
\sum\limits^{+l}_{m=-l}
q^{-2m}
=
\sum\limits^{+l}_{m=-l}
q^{2m}
\cong
b^{2l}_q
\cong
[2l+1]_q.
\end{align*}
The application of Proposition \ref{PROP:partial_derivative} completes the proof that $\extd$ extends. 
\end{proof}

\subsection{Generalising to other coquasitriangular Hopf algebras}
The bicovariant 4D calculus on $A=\C_q[SU_2]$ is an example of a canonical construction whenever $A$ is coquasitriangular in the dual of the sense of V.G. Drinfeld, i.e. a map $\CR:A\tens A\to \C$ obeying certain axioms. This gives a bicovariant calculus for any $L\subset A$ a subcoalgebra \cite{Majid2015}. We define 
\[ \CQ:A^+\tens L\to \C,\quad \CQ(a\tens b)=\CR(b\o\tens a\o)\CR(a\t\tens b\t)\]
 which we view as $\varpi=\CQ:A^+\to \Lambda^1=L^*$. If this is not surjective we take $\Lambda^1$ to be the image, but in examples it tends to be surjective so we suppose this as a property of the data $(A,\CR,L)$. In addition $L$ is canonically a left crossed $A$-module \cite{Majid2015} which makes $\Lambda^1$ a right crossed $A$-module with $\varpi$ a morphism. Here the left action on $L$ is
\[ a\la b=b\t\CR(b\o\tens a\o)\CR(a\t\tens b\th),\quad \forall a\in A,\ b\in L.\]

The simplest case of interest is when $L$ is the span of the matrix elements of a corepresentation $t\in \widehat{A}$,  $L={\rm span}\{t^\alpha{}_\beta\}$. We let $\{e_{\alpha\beta}\}$ be the dual basis of $\Lambda^1$, so $f^{\alpha\beta}=t^\alpha{}_\beta$ is the dual basis element to $e_{\alpha\beta}$. We let $e_\alpha$ be a basis of the corepresentation $V$ so $\Delta_R e_\alpha=e_\beta\tens t^\beta{}_\alpha$.  The associated left representation of any Hopf algebra $U$ dually paired to $A$ is $t(x)^\alpha{}_\beta=\<t^\alpha{}_\beta,x\>$ for all $x\in U$ or  $x.e_\alpha=e_\beta\<t^\beta{}_\alpha,x\>$. In this case if $\pi_{ij}$ are the matrix elements of a representation $\pi\in\widehat{A}$ then
\begin{equation} 
\label{EQ:symbol-4D}
\sigma^{\alpha\beta}(\pi)_{ij}=f^{\alpha\beta}(\varpi(\pi_{ij}-\delta_{ij}))=\CQ(\pi_{ij}\tens t^\alpha{}_\beta)-\delta_{ij}\delta^\alpha{}_\beta
\end{equation}
which we can usually write as 
\[ 
\sigma^{\alpha\beta}(\pi)_{ij}=\pi(x^{\alpha\beta})_{ij},
\quad \CQ(a\tens t^\alpha{}_\beta)=\eps(a)\delta^\alpha{}_\beta + \<a, x^{\alpha\beta}\>,\quad \forall a\in A,\]
for some elements $x^{\alpha\beta}\in U$ for suitable $U$. Here $x^{\alpha\beta}=((Sl^-)l^+)^\alpha{}_\beta-\delta^\alpha{}_\beta$ in the quantum groups literature \cite{Majid1995} for certain elements $l^\pm{}^\alpha{}_\beta\in U$. These elements are evaluated in the associated matrix representation $\pi$ of $U$ and $\eps(x^\alpha{}_\beta)=0$ is implied by the above. Similarly, the adjoint of the action on $L$ gives the right action
\[e_{\alpha\beta}\ra a=e_{\gamma\delta}\CR(t^\gamma{}_\alpha\tens a\o)\CR(a\t\tens t^\beta{}_\delta)=e_{\gamma\delta}\<a, Sl^-{}^\gamma{}_\alpha l^+{}^\beta{}_\delta\>.\]
Hence the action of matrix elements $\pi_{ij}$ of a corepresentation is
\[ e_{\alpha\beta}\ra\pi_{ij} = e_{\gamma\delta}\pi((Sl^-{}^\gamma{}_\alpha) l^+{}^\beta{}_\delta)_{ij},\]
or in terms of the matrix that governs the commutation relations, this is  
\begin{equation}
\label{EQ:sigma-alpha-beta-gamma-delta}
\sigma_{\alpha\beta}{}^{\gamma\delta}(\pi)_{ij}=\pi((Sl^-{}^\gamma{}_\alpha) l^+{}^\beta{}_\delta)_{ij}.
\end{equation} 
For the example of $U_q(su_2)$ one has \cite{Majid1995} 
\begin{equation}
\label{EQ:l-pm}
l^+= \begin{pmatrix}
q^{\frac{H}{2}} &  0\\
q^{-\frac 1 2}(q-q^{-1})X_+ &  q^{-\frac{H}{2}} 
\end{pmatrix},\quad 
l^-= \begin{pmatrix}
q^{-\frac{H}{2}} &  q^{\frac 1 2}(q^{-1}-q)X_- \\
0&  q^{\frac{H}{2}}
\end{pmatrix},
\end{equation}
giving the formulae for $x^\alpha{}_\beta$ previously used. It seems clear that this calculus will similarly extend to $C_\D^\infty(\G)$  for the general $q$-deformation of a compact simple group with $A=\C[\G]$ coquaistriangular. Details will be considered elsewhere.

\subsection{Concluding remarks}

Having a suitable summable $\D$ to define a smooth subspace $C^\infty_\D$ to which the differential calculus extends, as above, is an important step towards an actual geometric Dirac operator. In the coquasitriangular case with the bicovariant calculus defined by a matrix corepresentation, we have $\Lambda^1={\rm End}(V)$ for some comodule $V$ and following \cite{Majid2003} we can define `spinor sections' $S^\infty_\D=C^\infty_\D\tens V$  and a canonical map
\[ D:S^\infty_\D\to S^\infty_\D,\quad  D(s_\beta\tens e_\beta)=\del^{\alpha\beta}s_\beta\tens e_\alpha\]
where $\{e_\alpha\}$ is a basis of $V$ and $s_\alpha\in C^\infty_\D$. At the algebraic level this was $D=(\id\tens{\rm ev})(\extd\tens\id): A\tens V\to A\tens {\rm End}(V)\tens V\to A\tens V$ in \cite{Majid2003}, but since the partial derivatives extend to `smooth functions' we see that so does $D$ to our `smooth sections'.  This was studied at the algebraic level in detail for  $A=\C_q[SU_2]$ and justified as a natural Dirac-like operator that bypasses the Clifford algebra in the usual construction of the geometric Dirac operator, and fits with that after we add an additional constant curvature term (a multiple of the identity). Using our results for this quantum group we have 
\[ D\begin{pmatrix}\alpha t^l_{mn}\cr \beta t^{l'}_{pr}\end{pmatrix}=\begin{pmatrix}\alpha t^l_{ms}\sigma^{11}(t^l_{sn})+\beta t^{l'}_{ps}\sigma^{12}(t^{l'}_{sr})\cr \alpha t^l_{ms}\sigma^{21}(t^l_{sn})+\beta t^{l'}_{ps}\sigma^{22}(t^{l'}_{sr}) \end{pmatrix}\]
for coefficients $\alpha,\beta$ and for the symbols (\ref{EQ:sigma-x-alpha-beta}) given previously (we sum over $s$ in the appropriate range). The eigenvalues of the geometrically normalised $D/\lambda$ when restricted to both spinor components in the Peter-Weyl subspace spanned by $\{t^l_{mn}\}$ are 
\[ (i)\quad q^{l+1}[l]_q;\quad (ii) (l>0)\quad -q^{-l}[l+1]_q\]
and fully diagonalise this subspace of dimension $2(2l+1)^2$, and hence together full diagonalise $D$. The type (i) eigenvalues were already noted for the reduced Hopf algebras at odd roots unity in \cite[Prop.~5.2]{Majid2003} in the equivalent form $q^2[2l;q]/[2;q]=q^2[l;q^2]$, where $[m;q]=(q^m-1)/(q-1)=1+q+\cdots+q^{m-1}$. We see using our Fourier methods that we also have a second set (ii) both at roots of unity (beyond the 3rd root) and for real or generic $q$. These eigenvalues are not the $\pm [2l+1]_q$ that we might have expected naively deforming the operators $\mathcal{D}$ with eigenvalues $2l+1$ discussed in Example~\ref{EX:pal-1} but are in the same ball-park. Note that our geometric $D$ is not directly comparable to $\mathcal D$ in that section because our spinor space is two-dimensional so that $D$ does not act on one copy of the coordinate algebra, and nor should it geometrically. 

Also note that the bicovariant matrix block calculi are typically inner in the sense of a nonclassical direction $\theta$ such that $[\theta,f]=\lambda\extd f$, and that is the case for the 4D calculus on $\C_q[\SU2]$ with $\theta=e_a+e_d$. One can choose a more geometric basis $e_z,e_b,e_c,\theta$ where the first three have a classical limit as usual and $e_z=q^{-2}e_a-e_d$. The partial derivative $\del^\theta$ for the $\theta$-direction in this basis turns out to be the $q$-deformed Laplacian $\Delta_q$ as explained in \cite{Majid2015}. There is a quantum metric
\[ g=e_c\tens e_b+q^2 e_b\tens e_c+{q^2\over q+q^{-1}}(e_z\tens e_z-\theta\tens\theta)\]
and denoting its coefficients as $g_{ij}$ one has a natural $q$-Laplace operator\cite{Majid2015}
\[ \Delta_q={q\over 2}g_{ij}\del^i\del^j,\quad \del^\theta= {q\del^a+q^{-1}\del^d\over q+q^{-1}}={q^2\lambda^2 \over q+q^{-1}}\Delta_q\]
(where we have changed to our more geometric normalisation of $\del^i$ and $\extd$). Once again, since we have seen that the partial derivatives extend to $C_\D^\infty(\SUq2)$, this $\Delta_q$ also extends and, using our result (\ref{EQ:sigma-x-alpha-beta}), we are in a position to compute it 
in our Peter-Weyl basis as
\begin{align*} \Delta_qt^l_{mn}={[2]_q\over q^{2}\lambda^{2}}\del^\theta t^l_{mn}&=q^{-2}\lambda^{-2}(q t^l_{ms}\sigma^{11}(t^l{}_{sn})+q^{-1}t^l_{ms}\sigma^{22}(t^l{}_{sn}))\\ &={q(q^{2l}-1)+q^{-1}(q^{-2l}-1)\over (q-q^{-1})^2}t^l_{mn}\\ &=[l]_q[l+1]_qt^l_{mn}\end{align*}
for $m,n=-l,\cdots,l$. One could then take a square root involving $\Delta_q$ much as in Example~\ref{exDeltaG} for the operator $\D$ to provide the smoothness.  

Further $q$-harmonic analysis using our Fourier methods will be considered elsewhere to include smooth functions and harmonic analysis on the $q$-sphere obtained from the 3D differential calculus on $C^\infty_\D(\SUq2)$, extending the algebraic line for the geometric Dirac operator on the $q$-sphere in \cite{BeggsMajid2015}. Note  that our $q$-geometric Dirac operators are not exactly part of spectral triples in the strict Connes sense although the one on the $q$-sphere comes close at the algebraic level.

\newcommand{\etalchar}[1]{$^{#1}$}

\end{document}